\documentclass[12pt,reqno]{amsart}
\pdfoutput=1
\usepackage{etoolbox} % to patch the style for subsubsections
\patchcmd{\subsubsection}{\normalfont}{\bfseries\textnormal}{}{} % subsections as sections; boldfaced font

\usepackage[vmargin=2cm,hmargin=2cm]{geometry} 

\usepackage{amssymb, latexsym, color, etaremune, enumerate, tikz}
\usepackage{caption,float,textcomp,units,xfrac,url,tabularx}
\usepackage{placeins}

\usetikzlibrary{decorations.markings}
\usepackage{microtype} % improves the typographic algorithm of pdfTeX
\usepackage[english]{babel} % improves hyphenation
\newtheorem{Thm}{Theorem}

\newtheorem{Lem}{Lemma}

\newtheorem{Cor}{Corollary}
\newtheorem{Con}{Conjecture}
\newtheorem{Prop}{Proposition} 

\theoremstyle{definition}
\newtheorem{Rem}{Remark}

\newtheorem{Def}{Definition}
\newtheorem{Example}{Example}

\DeclareMathOperator{\lcm}{lcm} 
\DeclareMathOperator{\ord}{ord}

\newcommand{\Q}{\mathbb{Q}}

\newcommand{\mc}{\mathcal}

\makeatletter
\let\@@pmod\pmod
\DeclareRobustCommand{\pmod}{\@ifstar\@pmods\@@pmod}
\def\@pmods#1{\mkern4mu({\operator@font mod}\mkern 6mu#1)}
\makeatother

\makeatletter %  kronecker symbol macro; auto-adapts to display mode; usage: \kronecker{p}{q} in math mode; inline or as index
\newcommand*\kronecker[2]{%
  \relax\if@display
    \expandafter{(\frac{#1}{#2})}
  \else
    \expandafter{\bigl(\frac{#1}{#2}\bigr)}
  \fi
}
\makeatother
\begin{document}
\title[The ubiquity of Ramanujan-style congruences]{Prime divisors of $\ell$-Genocchi numbers and 
the ubiquity of Ramanujan-style congruences of level $\ell$}

\author[P. Moree]{Pieter Moree}
\address{Max-Planck-Institut f\"ur Mathematik, Vivatsgasse 7, D-53111 Bonn, Germany}
\email{moree@mpim-bonn.mpg.de}

\author[P. Sgobba]{Pietro Sgobba}
\address{University of Luxembourg, Department of Mathematics, 6 Avenue de la Fonte, L-4364 Esch-sur-Alzette, Luxembourg}
\email{pietro.sgobba@uni.lu}


\maketitle

\begin{abstract}
Let $\ell$ be any fixed prime number.
We define the
$\ell$-Genocchi numbers by 
$G_n:=\ell(1-\ell^n)B_n$, with $B_n$ the $n$-th Bernoulli number. 
They are integers.
We introduce and study a variant of Kummer's notion of regularity of primes. 
We say that an odd prime $p$ is $\ell$-Genocchi irregular 
if it divides at least one of the 
$\ell$-Genocchi numbers $G_2,G_4,\ldots, G_{p-3}$, 
and $\ell$-regular otherwise. 
With the help of techniques used
in the study of Artin's primitive root conjecture,
we give asymptotic estimates
for the number of $\ell$-Genocchi irregular primes in 
a prescribed arithmetic progression in case $\ell$ is odd.
The case $\ell=2$ was already dealt with
by Hu, Kim, Moree and Sha (2019).
\par  Using similar methods we study the prime factors
of
$(1-\ell^n)B_{2n}/2n$ and $(1+\ell^n)B_{2n}/2n$. 
This
allows us to estimate the number of primes $p\le x$
for which there exist modulo $p$ Ramanujan-style congruences  
between the Fourier coefficients of an Eisenstein series
and some cusp form of prime 
level $\ell$. 
\end{abstract}

\section{Introduction}
Recall that the $n$-th Bernoulli number $B_n$ 
 is  implicitly defined as the coefficient of 
$t^n$ in the generating function 
\begin{equation}
\label{Bgen}
\frac{t}{e^t-1}=\sum_{n=0}^{\infty}B_{n}\frac{t^n}{n!}. 
\end{equation}
The Bernoulli numbers are rational. 
It is easy to see that $B_0=1,B_1=-1/2$ and $B_{2n+1}=0$ for $n\ge 1$.
By the \emph{von Staudt-Clausen
theorem} (see for example \cite[Chp.\,3]{AIK}) the remaining Bernoulli numbers satisfy
\begin{equation}
\label{vsclausen}
B_{2n}+\sum_{p-1\mid 2n}\frac{1}{p}\in \mathbb Z,
\end{equation}
where the sum is over the primes $p$ for which $p-1$ divides $2n$, and
thus their denominators are well understood. However, their numerators are
far less so. We say that an odd prime $p$ is  \textit{$B$-irregular}  
if $p$ divides the numerator of at least one of the Bernoulli numbers $B_{2},B_{4},\ldots,B_{p-3}$, 
and \textit{$B$-regular} otherwise. 
This notion has an important application in algebraic number theory. Let $\mathbb{Q}(\zeta_{p})$ with $\zeta_p=e^{2\pi i/p}$ be the $p$-th cyclotomic field, 
and $h_p$ its class number. 
\begin{Thm}[Kummer]
\label{thm:Kummer}
Let $p$ be an odd prime.
If $p\nmid h_p$, then
the Fermat equation $x^p+y^p=z^p$ does not have 
a solution in positive integers
$x,y,z$ with $p$ coprime to $xyz$. 
An odd prime $p$ is $B$-regular if and only if $p\nmid h_p$.
\end{Thm}
This result of Kummer is considered to be one of the highlights of
19th century number theory. 
In the process of proving it, Kummer developed a lot of algebraic number
theory including the notion of ideals. For a more modern proof of 
Theorem \ref{thm:Kummer} using $p$-adic 
methods, see the book by Borevich and Shafarevich \cite{BS}.
\par In 1915, Jensen (as a student!) proved that there are infinitely many irregular
primes (see, e.g., \cite[p.\,381]{BS} 
or \cite[p.\,20]{Murty}).
Unfortunately the same is not known for regular primes, although numerical evidence indicates that
about 61\% of all primes are regular. This number is easily explained using an heuristical argument due to 
Siegel \cite{Siegel}. 
Assuming that the divisor
structure of Bernoulli denominators is random, we expect that 
$p\nmid B_k$ with probability $1-1/p$. We thus might expect that $p$ is regular with probability
$(1-1/p)^{(p-3)/2}$, which as $p$ gets large tends to $e^{-1/2}=0.60653\ldots$. 
\begin{Con}[Siegel]
\label{sheuristics}
The number of $B$-regular primes up to $x$ is given asymptotically by
\[ \frac{x}{\sqrt{e}\log x}. \]
\end{Con}
In this paper we study the divisibility of some sequences
related to Bernoulli numbers. As usual $\nu_p(a)$ denotes the
exponent of $p$ in the prime factorization of the rational number $a$.
\begin{Def}[Divisibility of a sequence]
Let $A=\{a_n\}_{n=1}^{\infty}$ be a sequence of rational numbers. We say that a prime $p$ divides the sequence $A$ if there exists $n\ge 1$ such that $a_n\ne 0$ and $\nu_p(a_n)\ge 1$. The set of prime
divisors of $A$ we denote by
${\mathcal Q}_A$. 
\end{Def}

Throughout, if $\mc S$ is a set of prime numbers, then we denote $\#\{p\le x:p\in \mc S\}$ by $\mc S(x)$.

It is a consequence of 
\eqref{vsclausen} and
the Kummer congruences, that a prime $p$ divides the sequence of Bernoulli numbers if and only if it is 
$B$-irregular. Recall that the \emph{Kummer congruence} (cf.\,Murty \cite[pp.\,18-19]{Murty}) 
implies that
\begin{equation}
    \label{Kummercongruence}
    \frac{B_j}{j}\equiv \frac{B_k}{k}\pmod*{p},~{\rm ~with~}j\equiv k\not\equiv 0\pmod*{p-1}.
\end{equation}

 The aim of this paper is to study the prime divisors of the
sequences $\{H_{2n}\}_{n=1}^{\infty}$, $\{H_{2n}^{-}\}_{n=1}^{\infty}$ and $\{H_{2n}^{+}\}_{n=1}^{\infty}$ with
\begin{equation}
\label{notation}    
H_{2n}: = (1-\ell^{2n})\frac{B_{2n}}{2n},\qquad H_{2n}^-:=(1-\ell^{n})\frac{B_{2n}}{2n}, \qquad  H_{2n}^+:=(1+\ell^{n})\frac{B_{2n}}{2n}. 
\end{equation}
In this context the notion of
irregularity will play an important role. 
\begin{Def}[(Ir)regularity]
Given a sequence of rational numbers $\{A_{k}\}_{k=1}^{\infty}$, we say that $p>3$ is $A$-regular if all of $\nu_p(A_2),\ldots,\nu_p(A_{p-3})$ are 
non-positive, and $A$-irregular otherwise. The prime $3$ is defined to be $A$-regular. The set of $A$-irregular primes is denoted by $\mc P_A$.
\end{Def}
Sometimes we use $H_{2n}^{-1}$ and $H_{2n}^{1}$ instead
of $H_{2n}^-$, respectively $H_{2n}^+$. Thus the final two entries in \eqref{notation} can be more compactly 
written as 
\[
H_{2n}^{\varepsilon}:=(1+\varepsilon\ell^{n})\frac{B_{2n}}{2n}, \qquad   
\varepsilon\in\{-1,1\}. 
\]
The number
$-H_k^{\varepsilon}/2$ occurs 
for even $k$ as constant term in the 
Fourier expansion of a generalization 
$E_{k,\ell}^{\varepsilon}(z)$
(given by \eqref{ES})
of the 
classical weight $k$ Eisenstein series $E_k(z)$ to the prime level $\ell$ 
setting. In case 
modulo a prime $p$ this constant is zero, the Eisenstein series 
$E_{k,\ell}^{\varepsilon}(z)$ is possibly coefficient wise congruent to a cusp form leading to a
congruence of Ramanujan type (such as \eqref{RC}).
In Sect.\,\ref{sec:hplusmincounting} we consider, given
 $\ell$ and $\varepsilon$, for how
many primes $p\le x$ there exists at least one Eisenstein
series 
$E_{k,\ell}^{\varepsilon}(z)$ such that 
modulo $p$ its constant term
$-H_k^{\varepsilon}/2$ is zero.
In the next section we discuss this modular form connection in more detail.

\subsection{Ramanujan style congruences for prime level $\ell$}
	Let $E_k$ be the Eisenstein series of even weight $k\ge  2$ for the group $SL_2({\mathbb Z})$,
	normalized so that its Fourier series expansion is
	$$E_k(z)=-\frac{B_k}{2k}+\sum_{n=1}^{\infty}\sigma_{k-1}(n)e^{2\pi i n z},$$
	where $\sigma_r(n) = \sum_{d|n}d^{r}$ is the
	$r$-th sum of divisors function.
	The prototype of a Ramanujan congruence goes back 
	to 1916 and asserts that
	\begin{equation}\label{RC}
		\tau(n)\equiv \sigma_{11}(n)  \pmod*{691},
	\end{equation}
	for every positive integer $n$. This can be viewed as a (coefficient-wise) congruence between the unique cusp form $\Delta(z)=\sum_{n=1}^{\infty}\tau(n)e^{2\pi i n z}$ of weight $12$ and the Eisenstein series 
	$E_{12}(z)$, namely
	$	\Delta \equiv E_{12}  \pmod*{691}$.
	Note that $691$ divides $B_{12}$.
	There are several well-known ways to prove, interpret, and 
	generalize this. Here we will only focus on a generalization to prime level
	$\ell$, where the associated Eisenstein series,
	for $\varepsilon\in \{\pm 1\}$, even weight $k\ge 2$ and prime level $\ell$, 
	is
	\begin{equation}\label{ES}
		E_{k,\ell}^{\varepsilon}(z)=E_k(z)+{\varepsilon}\ell^{k/2}E_k(\ell z).
	\end{equation}
	Notice that the constant term in the Fourier series of $E_{k,\ell}^{\varepsilon}(z)$ equals
	$-H_k^{\varepsilon}/2$.
	Kumar et al.\,\cite{KKMK} recently established the following result.
Throughout this article, given a rational number $a$, by $p\mid a$, we mean that the prime number $p$ 
divides the reduced numerator of $a$, that is $\nu_p(a)\ge 1$. 
\begin{Thm}\label{eigenform_varepsilon}
	Let $k$ be an even natural integer, $\ell$ and $p$ be primes. Let
$S_k^{\varepsilon}(\ell)$ denote the $\varepsilon$-eigenspace of the Atkin-Lehner operator $W_p$ inside $S_k(\ell)$, the space of modular 
cusp forms of weight $k$ for the group $\Gamma_0(\ell)$.\\
{\rm i)}	Suppose that $p\ge 5$ divides  $H_k$ for some even integer $k\ge 4$.
	Then there exists $\varepsilon \in \{\pm  1\}$ and a normalized eigenfunction $f \in S_k^\varepsilon(\ell)$ for all Hecke
	operators $T_q$ with $q\neq \ell$ a prime, and a prime ideal $\mathfrak p$ over $p$ in the coefficient field of $f$ such that %for 	any integer $n$ coprime to $p$, we have
		\begin{equation}
	\label{eq:Rtype}    
	f\equiv E_{k,\ell}^{\varepsilon} \pmod*{\mathfrak p}.
	\end{equation}
	{\rm ii)} Let $\varepsilon \in \{\pm  1\}$ be fixed. 
	Suppose that $p\ge 5$ divides $H_k^{\varepsilon}$ for some even integer
	$k\ge 4$.
	%$\frac{B_k}{k}(\ell^{k/2}+\varepsilon)$.
	Then there exists a normalized eigenfunction $f \in S_k^\varepsilon(\ell)$ for all Hecke
	operators $T_q$ with $q\neq \ell$ a prime, and a prime ideal $\mathfrak p$ over $p$ in the coefficient field of $f$ such that %for 	any integer $n$ coprime to $p$, we have
	\eqref{eq:Rtype} holds.  
\end{Thm}
If we fix a prime $\ell$, we can wonder about the ubiquity of the primes 
$p$ for
which a Ramanujan-congruence \eqref{eq:Rtype} for some even integer $k\ge 4$ exists. 
This amounts to estimating the number of prime
divisors $p\le x$ as $x$ gets large of the sequences $\{H_{2n}\}_{n=1}^{\infty}$ and $\{H_{2n}^{\varepsilon}\}_{n=1}^{\infty}$\footnote{For various reasons we included the terms
with $n=1$ as well. It is a consequence of the Kummer congruences that these
sequences have the same prime divisors as the ones with the term $n=1$ left out.}.
\par It is not so difficult to show (see Lemma \ref{lem:wieferich}) that a prime $p$ divides the
$H$-sequence if and only
if it is $H$-irregular or is in the Wieferich set
\begin{equation}\label{wieferich1}
    {\mathcal W}_{\ell}=\{p>2:\ell^{p-1}\equiv 1\bmod{p^2}\}.
\end{equation}
Likewise, $p$ divides the $H^{\varepsilon}$-sequence
if and only
if it is $H^{\varepsilon}$-irregular or is in the Wieferich set ${\mathcal W}_{\ell}$ (if $\varepsilon=-1$) or
${\mathcal W}_{\ell}^+$ (if $\varepsilon=1$),
where
\begin{equation}\label{wieferich2}
    {\mathcal W}_{\ell}^{\varepsilon}=\{p>2:\ell^{(p-1)/2}\equiv -\varepsilon\bmod{p^2}\},\quad \varepsilon\in\{-1,1\}.
\end{equation}
Note that ${\mathcal W}_{\ell}={\mathcal W}_{\ell}^-+{\mathcal W}_{\ell}^+$.
The Wieferich sets are believed to be very sparse and thus the congruence ubiquity problem in essence amounts to estimating ${\mathcal P}_{H^{\varepsilon}}(x)$, the number of
$H^{\varepsilon}$-irregular primes up to $x$. The results (too lengthy to be stated
here) are presented in Sect.\,\ref{sec:hplusmincounting}. We show that 
\begin{equation}\label{delta1}
    {\mathcal P}_{H^{\varepsilon}}(x)>\delta_1\frac{x}{\log x},\quad \delta_1>0,\quad x\rightarrow \infty,
\end{equation}
and making heuristical assumptions similar to that of Siegel, we conjecture that
$${\mathcal P}_{H^{\varepsilon}}(x)\sim \delta_2\frac{x}{\log x},\quad \delta_2>0,\quad x\rightarrow \infty,$$
with $\delta_1$ and $\delta_2>\delta_1$ explicit constants. Our constants $\delta_2$ are consistent with numerical data, see Table \ref{tabHpm}. The inequality \eqref{delta1} can be
compared to the best-known result for Bernoulli numbers.
Namely, Luca, Pizarro-Madariaga, and Pomerance \cite{Luca}
showed that the number ${\mathcal P}_B(x)$ of  irregular primes up to
$x$ satisfies
%\begin{equation}
%\label{eq:Luca}
$$
{\mathcal P}_B(x)  \ge (1+o(1)) \frac{\log\log x}{\log\log\log x}.
$$
%\end{equation}
\par In Sec.\,\ref{hahaha} we extend the above results by restricting to primes in
a prescribed arithmetic progression.
 Given $1\leq a<d$ coprime integers and a prime $\ell$, we define 
\[ \mc P_{H^{\varepsilon}}(d,a):= \{ p : p\equiv a \bmod d \text{ and } p \text{ is $H^{\varepsilon}$-irregular} \},\,\varepsilon\in\{-1,1\}. \]
In order to study the primes in these sets we consider the
following related sets:
\[ \mc A_{d,a}=\{p\equiv a\bmod d: \ord_p(\ell)=p-1  \},\,\mc A_{d,a}^-:=\{p\equiv a\bmod d:\ord_p(\ell)=(p-1)/2\}.\]
To these sets we associate infinite sums $\alpha_{d,a}$ and
$\alpha^-_{d,a}$ given by \eqref{alphada}, respectively 
\eqref{alphada-}. Under GRH it follows from general results of 
Lenstra \cite{Lenstra} that these are the respective relative (inside the set of primes $p\equiv a\bmod d$) densities of the sets.
The infinite sums $\alpha_{d,a}$ and
$\alpha^-_{d,a}$ are given in Euler product form in 
Theorem \ref{thm-pietersap} (well-known), respectively Theorem \ref{thm_Hminus} (new). 
Both are rational multiples of the \emph{Artin constant} 
\begin{equation}
\label{Artinconstantdef}
A= \prod_{\textrm{prime $p$}}\Big (1-\frac{1}{p(p-1)} \Big) = 0.3739558136192022880547280543464\ldots,
\end{equation}
The set $\mc A_{d,a}$ has been well-studied, but not so
$\mc A^-_{d,a}$, although some special cases occur in various number theoretical problems, see, e.g., 
\cite{CM,Daileda,MZ}.
\begin{Thm}
\label{thm:HinAP}
Let $\ell$ be an odd prime.
For $\epsilon>0$ arbitrary and 
$\varepsilon\in \{-1,1\}$ we have
\[ \mc P_{H^{\varepsilon}}(d,a)(x)\geq (1 - \delta^{\varepsilon}_{d,a}-\epsilon)\frac{x}{\varphi(d)\log x}, \]
where $\delta^-_{d,a}=\alpha_{d,a}+\alpha^-_{d,a}$. 
We have $\delta^+_{d,a}=\alpha_{d,a}+\rho_{\ell,1}(a,d)$,
where $\rho_{\ell,1}(a,d)$ 
is the density of prime divisors $p\equiv a\bmod{d}$ 
of 
the sequence $\{\ell^n+1\}_{n\ge 1}^{\infty}$. 
This is always a rational number 
and is explicitly determined in 
Moree and Sury \cite{MS}.
\end{Thm}
The proof is a quite immediate consequence of the 
results proved in Section 
\ref{hahaha} and similar to that of
Theorem \ref{thm:GinAP} given in
Section \ref{sec:thm3}. The details are left to the
interested reader.
\subsection{Counting $\ell$-Genocchi irregular primes}
We let $G=\{G_{2n}\}_{n\ge 1}$ with $G_{2n}:=2\ell n H_{2n}=\ell(1-\ell^{2n})B_{2n}$ be the 
sequence of \emph{$\ell$-Genocchi numbers}.
The inclusion of the factor $\ell$ ensures that they are integers.
In case $\ell=2$ we speak about \emph{Genocchi numbers}. These show
up in a result similar to Kummer's Theorem 
\ref{thm:Kummer}, see Hu and Kim \cite{HK}, and it might thus be reasonable to consider them also for arbitrary $\ell$.
An odd prime $p$ is said to be \textit{$\ell$-Genocchi (ir)regular} 
if and only if it is $G$-(ir)regular.
The first twenty $2$-Genocchi irregular primes are  
$$
17, 31, \textbf{37}, 41, 43, \textbf{59}, \textbf{67}, 73, 89, 97, \textbf{101}, \textbf{103}, 109, 113, 127, \textbf{131}, 137, 
\textbf{149}, 151, \textbf{157},
$$
where those that are
also $B$-irregular are put in boldface. There is a considerable literature
on the classical Genocchi numbers (Sect.\,\ref{subsec:Earlier2}), but little seems to have been done in the general
case (Sect.\,\ref{subsec:Earlierodd}). The $\ell$-Genocchi irregular primes show up in the study of the prime
divisors of $H$. Namely, a prime $p$ divides the $H$-sequence if and only if it is 
$\ell$-Genocchi irregular or in the Wieferich set ${\mathcal W}_{\ell}$ (Proposition \ref{prop:Hdivisor}).
As the Wieferich set is believed to be sparse, the study of the prime divisors
of $H$ is in essence that of $\ell$-Genocchi irregular primes.
\par Very little is known about the distribution of irregular primes in a 
prescribed arithmetic
progression. We will show that for the $\ell$-Genocchi irregular primes 
the situation is rather different. This is a consequence of the 
divisor structure of $\ell^{2n}-1$ being much better
understood than that of the Bernoulli numerators.
\par Given $1\leq a<d$ coprime integers and a prime $\ell$, we define 
\[ \mc P_{G}(d,a):= \{ p : p\equiv a \bmod d \text{ and } p \text{ is $\ell$-Genocchi irregular} \}. \]
We are interested in the behavior of $\mc P_{G}(d,a)(x)$ as $x$ gets large as
compared to that of $\pi(x;d,a):=\#\{ p\le x: p\equiv a \bmod d \}$, which is known to behave asymptotically as $x/(\varphi(d) \log x)$. The latter
function appears  in Theorem \ref{thm:GinAP} and 
Conjecture \ref{con:PinAP} and can thus be replaced by $\pi(x;d,a)$.
By $\kronecker{*}{*}$ we denote the Jacobi symbol.
\begin{Thm}
\label{thm:GinAP}
Let $\ell$ be an odd prime.
For $\epsilon>0$ arbitrary we have
\[ \mc P_{G}(d,a)(x)\geq (1 - \delta_{d,a}-\epsilon)\frac{x}{\varphi(d)\log x}, \]
where $\delta_{d,a}$ is given by \eqref{deltada} and worked out in Euler product form in 
Theorem \ref{thm-ap}.
We have $1-\delta_{d,a}>0$.
Moreover, either $1-\delta_{d,a}\ge 1/4$, or $4\mid d$ and 
$a\equiv 3\bmod4$, or $\ell\mid d$ and $\kronecker{a}{\ell}=-1$.
\end{Thm}
A similar result for $\ell=2$ can be found in Hu et al.\,\cite{HKMS}.
The relative density $1-\delta_{d,a}$ 
can be arbitrarily close to $0$, respectively $1$.
%Likewise, $1-\delta_{d,a}$ can be 
%arbitrarily close to one. 
For particulars see Sect.\,\ref{sec:extremal}.
\par Regarding the true behavior of 
$\mc P_{G}(d,a)(x)$ we make the following conjecture 
(consistent with numerical data, see Table \ref{tabG3}).
\begin{Con}
\label{con:PinAP}
Let $\ell$ be an odd prime.
Asymptotically one has
\[ \mc P_{G}(d,a)(x)\sim \bigg(1- \frac{\delta_{d,a}}{\sqrt{e}} \bigg)\frac{x}{\varphi(d)\log x}, \]
where $\delta_{d,a}$ is given by \eqref{deltada} and 
worked out in Euler product form in 
Theorem \ref{thm-ap}.
\end{Con}
A similar conjecture can be formulated for $\mc P_{H^{\varepsilon}}(d,a)(x)$, where one merely replaces
$\delta_{d,a}$ by $\delta^{\varepsilon}_{d,a}$.
\par By Theorem \ref{thm-ap} we have 
for any odd prime $\ell$ that $\delta_{d,a}=0$ if and only if $4\ell$ divides $d$, $\kronecker{a}{\ell}=1$ and $a\equiv1\bmod4$.
Thus Conjecture \ref{con:PinAP} leads to the following 
weaker conjecture.
\begin{Con}
Let $\ell$ be an odd prime.
The set of $\ell$-Genocchi regular primes in the primitive residue
class $a\bmod{d}$ has a positive density, provided we are not in the
case where $4\ell$ divides $d$, $\kronecker{a}{\ell}=1$ and $a\equiv1\bmod4$.
If $\ell=2$ the density is positive provided we are not in the case where $8$ divides $d$ and 
$a\equiv1\bmod8$.
\end{Con}
The case $\ell=2$ is not covered by our argumentation, but is
Conjecture 1.16 of Hu et al.\,\cite{HKMS}.

\section{Preliminaries}
\subsection{Basic properties of the numbers $H_n$}
{}From \eqref{vsclausen} we infer that $G_{2n}$ is an integer and
$H_n$ an $\ell$-integer (that is
satisfies $\nu_p(H_{2n})\ge 0$ for every prime $p\ne \ell$). 
Put $\zeta_{\ell}=e^{2\pi i/\ell}$.
Using \eqref{Bgen} we see that 
\begin{equation}
\label{generating}
 t\sum^{\ell-1}_{a=1} \frac{\zeta_\ell^a}{\zeta_\ell^a-e^t}=
\frac{t}{e^t-1}-
\frac{\ell t}{e^{\ell t}-1}= \sum_{n=1}^{\infty} \frac{H_nt^n}{(n-1)!}, 
\end{equation}
where the first identity follows on noting that, as formal series,
\begin{align*}
    \sum^{\ell-1}_{a=1} \frac{\zeta_\ell^a}{\zeta_\ell^a-e^t} & 
    = \sum_{n=0}^\infty \left( \sum^{\ell-1}_{a=1} \zeta_\ell^{-an} \right)e^{tn} 
    = -\sum_{\substack{n\geq 0\\ \ell\nmid n}} e^{tn}+(\ell-1)\sum_{\substack{n\geq0\\ \ell\mid n}}e^{tn} \\
    &=-\sum_{n=0}^\infty e^{tn} +\ell \sum_{n=0}^\infty e^{\ell tn}
    = \frac{1}{e^t-1} - \frac{\ell}{e^{\ell t}-1}.
\end{align*}
\par If $p-1\nmid 2n$, then Voronoi's congruence (see, for example, 
Murty \cite[Chp.\,1]{Murty}) gives
\begin{equation}
    \label{Voronoi}
H_{2n}\equiv -\ell^{2n-1}\sum_{j=1}^{p-1}j^{2n-1}\big[\frac{j\ell}{p}\big]\pmod*{p},
\end{equation}
where $[y]$ denotes the greatest integer function. 
\subsection{Divisibility of $H,H^-,H^+$ and 
the $\ell$-Genocchi numbers: elementary observations}

The following trivial result will play an important role. 
By  $\ord_{p}(\ell)$ we denote the multiplicative order of $\ell$ modulo $p$.
\begin{Lem}\label{lem_even}
Let $p$ and $\ell$ be distinct primes. \\
{\rm 1)} The prime $p$ divides
$\ell^n+1$ for some
integer $n\ge 1$ if and only if $\ord_{p}(\ell)$ is even.\\
{\rm 2)} The prime $p$ divides $\ell^n+1$ for some $1\leq n\leq(p-3)/2$ if and only if 
$\ord_{p}(\ell)$ is even and not equal to $p-1$. \\
{\rm 3)} The prime $p$ divides $\ell^n+1$ for some $1\leq n\leq p-2$ 
with
$n\ne (p-1)/2$ if and only if 
$\ord_{p}(\ell)$ is even and not equal to $p-1$. \\
{\rm 4)} The prime $p$ divides $\ell^n-1$ for some $1\leq n\leq(p-3)/2$ if and only if 
$\ord_{p}(\ell)<(p-1)/2$.\\
{\rm 5)} The prime $p$ divides $\ell^n-1$ for some $1\leq n\leq p-2$ with
$n\ne (p-1)/2$ if and only if $\ord_{p}(\ell)<(p-1)/2$.\\
{\rm 6)} The prime $p$ divides $\ell^{2n}-1$ for some $1\leq n\leq(p-3)/2$ if and only if 
$\ord_{p}(\ell^2)<(p-1)/2$.
%{\rm 2)} The prime $p$
\end{Lem}
\begin{proof}
Left to the reader
(cf.\,Moree \cite[Prop.\,2]{Mordivisors}). 
\end{proof}
\begin{Rem}
\label{rem:odd}
%Note that $\ord_p(\ell^2)=\ord_p(\ell)$ if $\ord_p(\ell)$ is odd
%and $\ord_p(\ell)/2$ otherwise.
We will use various times the trivial observation that 
\begin{equation}
    \label{eq:trivial}
\ord_p(\ell^2)=
\begin{cases} 
\ord_p(\ell) & \text{if }2\nmid \ord_p(\ell);\\
\ord_p(\ell)/2  & \text{otherwise}.
\end{cases}
\end{equation}
%It follows from this lemma  that an odd prime $p$ is both H$^-$- and H$^+$-regular if and only if $p$ is $B$-regular and either $\ord_p(\ell)=p-1$ or $p\equiv 3\bmod 4$ and $\ord_p(\ell)=(p-1)/2$.
\end{Rem}
%Thus the primes $p$ that occur in the above are all those for
%which $\ord_{\ell}(p)$ is even and not equal to $p-1$.
\subsubsection{The $H$-sequences}
With the help of Lemma \ref{lem_even} we
will now characterize $H$-, $H^{-}$- and $H^{+}$-irregular primes.
\begin{Lem}\label{lem-H}
Let $p\ne \ell$ be an 
odd prime. \\
{\rm 1)} It is $H$-irregular if and only if it is $B$-irregular or $\ord_p(\ell^2)
<(p-1)/2$. \\
{\rm 2)} It is $H^-$-irregular if and only if it is $B$-irregular or $\ord_p(\ell)
<(p-1)/2$. \\
{\rm 3)} It is $H^+$-irregular if and only if it is $B$-irregular or 
$\ord_p(\ell)$ is even and not equal to $p-1$.\\
{\rm 4)} It is $H$-irregular if and only if it is either  $H^-$- or  $H^+$-irregular.\\
{\rm 5)} It is both $H^-$- and  $H^+$-regular if it is $B$-regular and satisfies $p\equiv 3\bmod 4$ and $\ord_p(\ell)=(p-1)/2$. 
\end{Lem}
\begin{proof}
The claims are clearly true for $p=3$,  and
so we may assumue $p>3$.
It follows from \eqref{vsclausen} that the prime factors of the denominator of $B_{2n}$ are precisely those primes $p$ such that $p-1$ divides $2n$. Therefore, 
\begin{equation}
\label{eq:gelijk}
\nu_p(B_{2n})=\nu_p(B_{2n}/2n),\quad 1\leq n\leq (p-3)/2.
\end{equation}
Suppose that $p$ is $H$-irregular, i.e.\ $\nu_p(H_{2n})\ge 1$ for 
some $1\leq n\leq (p-3)/2$. 
By \eqref{eq:gelijk} and Lemma 
\ref{lem_even}.6 this is equivalent with $p$ being $B$-irregular
or $\ord_{p}(\ell^2)<(p-1)/2$.
\par The proofs of 2) and 3) are analogous, and follow by a similar argument involving
Lemma \ref{lem_even}. Part 4) is a consequence of the observation that $p\mid \ell^{2n}-1$ 
if and only if $p\mid \ell^n-1$  or $p\mid \ell^n+1$.  Finally, part 5) follows from parts
2) and 3) on taking into account the identity \eqref{eq:trivial}.
\end{proof}
\begin{Lem}
\label{lem:Wieferich}
Let $p$ be an odd prime and $n$ a positive integer
such that $p-1$ divides $2n$. Then we have
$$\nu_p(H_{2n})=\nu_p(\ell^{p-1}-1)-1.$$
Further,
$$
\nu_p(H_{2n}^-)=
\begin{cases}
\nu_p(\ell^{p-1}-1)-1 & \text{if~}(p-1)\mid n;\\
\nu_p(\ell^{(p-1)/2}-1)-1 & \text{if~}(p-1)\nmid n\text{~and~}\kronecker{\ell}{p}=1;\\
-1-\nu_p(n)  & \text{if~}(p-1)\nmid n\text{~and~}\kronecker{\ell}{p}=-1.
\end{cases}
$$
Also,
$$
\nu_p(H_{2n}^+)=
\begin{cases}
\nu_p(\ell^{(p-1)/2}+1)-1 & (p-1)\nmid n\text{~and~}\kronecker{\ell}{p}=-1;\\
-1-\nu_p(n)  & otherwise.
\end{cases}
$$
\end{Lem}
\begin{proof}
Write $2n=(p-1)p^em$, with $p\nmid m$. Then, using \eqref{vsclausen} and
the elementary observation that if $a\equiv 1\bmod p$ and $a\ne 1$, then 
$\nu_p(a^{j}-1)=\nu_p(a-1)+\nu_p(j)$ (cf.\,Beyl \cite{Beyl}), we find that
$$\nu_p(H_{2n})=\nu_p(\ell^{2n}-1)+\nu_p(B_{2n})-\nu_p(2n)=\nu_p(\ell^{p-1}-1)+e-1-e=\nu_p(\ell^{p-1}-1)-1.$$
The two remaining statements are proved similarly, with
the difference that whereas $\nu_p(\ell^{2n}-1)>1$, it can happen 
that $\nu_p(\ell^{n}+\varepsilon)=0$. For the final claim we use that if 
$a\equiv -1\bmod p$ and $a\ne -1$ and $j$ is even, then 
$\nu_p(a^{j}-1)=\nu_p(a+1)+\nu_p(j)$.
\end{proof}
Recall the definitions \eqref{wieferich1} and \eqref{wieferich2} of the Wieferich sets.
\begin{Lem}
\label{lem:wieferich}
Let $\ell$ and $p>2$ be distinct primes. Then\\
{\rm 1)} $p$ divides the $H$-sequence if and only if
it is $H$-irregular or is in the Wieferich set 
${\mathcal W}_{\ell}$;\\
{\rm 2)} $p$ divides the $H^-$-sequence if and only if
it is $H^-$-irregular or is in the Wieferich set 
${\mathcal W}_{\ell}$;\\
{\rm 3)}  $p$ divides the $H^{+}$-sequence if and only if
it is $H^{+}$-irregular or is in the Wieferich set 
${\mathcal W}_{\ell}^{+}$.
\end{Lem}
\begin{proof}
We will only
deal with cases 2) and 3), since case 1) is similar and easier.
It is not difficult to see
that only finitely many terms of the $H^{\varepsilon}$-sequence have to be considered in order to decide whether $p$ divides the sequence
or not. 
Indeed, the Kummer congruence \eqref{Kummercongruence} implies
that if $p-1\nmid 2n$, then
\[
    \frac{B_{2n+r(p-1)}}{2n+r(p-1)}\Big(1+\varepsilon \ell^{n+r(p-1)/2}\Big)\equiv 
    \frac{B_{2n}}{2n}\Big(1+\varepsilon\kronecker{\ell}{p}^r\ell^{n}\Big)\pmod*{p}.
\]
%Due to the condition $k\ge 4$ instead of $k\ge 2$ in Theorem 
%\ref{eigenform_varepsilon}, these
%sets differ by at most finitely many primes from the sets we are
%interested in.
Thus we have periodicity modulo 
$p-1$ if $\kronecker{\ell}{p}=1$ and modulo
$2(p-1)$ otherwise. 
In particular, if $(p-1)\nmid 2n$ it is enough to consider the $p$-divisibility of
$$H_2^{\varepsilon},H_4^{\varepsilon}\ldots,H_{p-3}^{\varepsilon},H_{p+1}^{\varepsilon},
H_{p+3}^{\varepsilon},\ldots,H_{2p-4}^{\varepsilon}$$
The case $p-1\mid 2n$ is not covered, but
for this we invoke 
Lemma \ref{lem:Wieferich}, which shows that in this
case $p$ divides 
the $H^-$-sequence if and only $p$ is in 
 ${\mathcal W}_{\ell}$ and divides the 
 $H^+$-sequence if and only $p$ is in 
 ${\mathcal W}_{\ell}^+$.
So we restrict
now to the case $p-1\nmid 2n$.
\par 2) Now $\varepsilon=-1$.
By Lemma \ref{lem_even}.5 the prime $p$ is 
a divisor if and only if it is $B$-irregular or 
$\ord_{p}(\ell)<(p-1)/2$. By Lemma \ref{lem-H}.2 this is equivalent with
$p$ being $H^{-}$-irregular.
\par 3) Now $\varepsilon=1$. By Lemma \ref{lem_even}.3 the prime $p$ is 
a divisor if and only if it is $B$-irregular or 
$\ord_{p}(\ell)$ is even and not equal to $p-1$. By 
Lemma \ref{lem-H}.3 this is equivalent with
$p$ being $H^{+}$-irregular.
\end{proof}
\subsubsection{The $\ell$-Genocchi numbers}
Since $H_n=G_n/\ell n$, it is natural to wonder how $H$-(ir)regular and
$\ell$-Genocchi (ir)regular primes are related.
\begin{Prop}
\label{prop:HvG}
An odd prime $p\neq\ell$ is $H$-irregular if and only if $p$ is $\ell$-Genocchi irregular. If $\ell>3$, then $\ell$ is $\ell$-Genocchi irregular, and $\ell$ is $H$-irregular if and only if it is $B$-irregular.
\end{Prop}

\begin{proof}
If $p\neq\ell$, then for every $n=1,2,\ldots,(p-3)/2$, since $p$ and $p-1$ do not divide $2n$, we have that $\nu_p(H_{2n})=\nu_p(G_{2n})$.
If $\ell>3$, then for every $n=1,2,\ldots,(\ell-3)/2$, since $\ell$ and $\ell-1$ do not divide $2n$, we have that $\nu_\ell(\ell(1-\ell^{2n})B_{2n})\geq1$, and $\nu_\ell((1-\ell^{2n})B_{2n}/2n)=\nu_\ell(B_{2n})$.
\end{proof}
In contrast, the prime divisor structure of $G$- and $H$-sequences are rather
different.
\begin{Prop}
Let $\ell$ be a fixed prime. Given any prime $p$ (for $\ell\le 3$ we suppose $p\neq\ell$), there exists an integer $n$ such that $v_p(G_{2n})\geq1$. If $\ell\le 3$, then $v_\ell(G_{2n})= 0$ for all $n$.
\end{Prop}
\begin{proof}
If $p\neq\ell$, we take $2n=p(p-1)$. Then $\ell^{2n}\equiv1\bmod p^2$ and $\nu_p(B_{2n})=-1$, so that $\nu_p(G_{2n})\geq1$.
If $p=\ell$ (and $\ell>3$), then we take $n=(\ell-3)/2$ and find that
$\nu_{\ell}(G_{2n})\ge 1$.
If $\ell\le 3$, then  $\ell-1\mid 2n $ and
hence $\nu_\ell(B_{2n})=-1$, so that $\nu_\ell(G_{2n})=0$.
\end{proof}

\begin{Prop}
\label{prop:Hdivisor}
A prime $p$ divides the $H$-sequence if and only if it is 
$\ell$-Genocchi irregular or in the Wieferich set ${\mathcal W}_{\ell}$.
\end{Prop}
\begin{proof}
This follows on combining Lemma \ref{lem:wieferich}.1 and
Proposition \ref{prop:HvG}.
\end{proof}

\begin{Prop}\label{lem_Greg}
Let $\ell$ be  a prime. An odd prime $p$, with $p\neq\ell$, is $\ell$-Genocchi regular if and only if it is $B$-regular and $\ord_p(\ell^2)=(p-1)/2$.
\end{Prop}

\begin{proof}
For $\ell=2$ it was shown by  Hu et 
al.\,\cite[Theorem 1.8]{HKMS}. 
For the remaining $\ell$ it follows from Proposition \ref{prop:HvG} and Lemma \ref{lem-H}.1.
\end{proof}
\begin{Cor}
\label{cor:simple}
Let $\ell$ be a prime.
If $p\equiv 1\bmod{4}$ and $\kronecker{\ell}{p}=1$, then $p$ is 
$\ell$-Genocchi irregular.
\end{Cor}
\begin{proof}
By contradiction. Suppose that $\ord_p(\ell^2)=(p-1)/2$.
The assumption $p\equiv 1\bmod{4}$ ensures that $\ord_p(\ell)=p-1$, contradicting
the assumption that $\kronecker{\ell}{p}=1$.
\end{proof}
\begin{Prop}
Let $\ell$ be a prime. If $p$ is congruent to an odd square modulo $4\ell$,
then $p$ is $\ell$-Genocchi irregular.
\end{Prop}
For $\ell=2$ one verifies this directly.
For odd $\ell$ it follows 
from Corollary \ref{cor:simple} on making use of an alternate form 
of the  law of
quadratic reciprocity, initially conjectured by Euler, cf.\,Cox \cite[p.\,15]{Cox}.
\begin{Lem}
If $p$ and $q$ are distinct odd primes, then
$\kronecker{q}{p}=1$ if and only if $p\equiv \pm \beta^2 \bmod{4q}$ for some
odd integer $\beta$.
\end{Lem}

\section{Further preliminaries}
We recall some relevant results and conjectures from the literature.

\subsection{Primitive roots in arithmetic progression} 
\label{sec:thmap}
Put
\[ \mc A_{d,a}=\{p:p\equiv a\bmod d, \, \ord_p(\ell)=p-1  \}.\]
Under GRH the natural density of this set 
is given by 
\begin{equation}
    \label{alphada}
\alpha_{d,a} = \sum_{n=1}^\infty\frac{\varphi(d)\mu(n)c_a(n)}{[\Q(\zeta_{[d,n]},\ell^{1/n}):\Q]},
\end{equation}
where $c_a(n)=1$ if the automorphism $\sigma_a$ of $\Q(\zeta_d)$ determined by $\sigma_a(\zeta_d)=\zeta_d^a$ is the identity on the field $\Q(\zeta_d)\cap\Q(\zeta_{n},\ell^{1/n})$, and $c_a(n)=0$ otherwise.
Its Euler product form was first evaluated
by Moree \cite{MOa} for 
arbitrary $\ell$. In the relevant case for
us where $\ell$ is an odd prime, this result takes on a rather simpler form.
\begin{Thm}\label{thm-pietersap}
Let $a$ and $d$ be coprime integers and $\ell$ be an odd prime.
Put
\[ \alpha_{d,a}=A\, c_1(d,a)\, R(d,a), \]
with
\begin{equation}
\label{eq:rda}
R(d,a) = \prod_{p\mid(a-1,d)}\left( 1-\frac{1}{p} \right) \prod_{p\mid d}\left( 1+\frac{1}{p^2-p-1} \right) 
\end{equation}
and
\[ c_1(d,a) = \begin{cases} 
1-\kronecker{\ell}{a} & \text{if } \ell\equiv 1\bmod 4,\ell\mid d;\\
 1+\frac{1}{\ell^2-\ell-1}  & \text{if } \ell\equiv 1\bmod 4,\ell\nmid d;\\
 1-\kronecker{\ell}{a} & \text{if } \ell\equiv 3\bmod 4,4\mid d,\ell\mid d;\\
 1+\kronecker{-1}{a}\frac{1}{\ell^2-\ell-1}  & \text{if } \ell\equiv 3\bmod 4,4\mid d,\ell\nmid d;\\
1 & \text{if } \ell\equiv 3\bmod 4, 4\nmid d.
\end{cases} \]
Let $\epsilon>0$ be any fixed real number. Then, for every $x$ sufficiently large,
\begin{equation}
    \label{eq:upperap}
\mc A_{d,a}(x)\leq (\alpha_{d,a}+\epsilon)\frac{x}{\varphi(d)\log x}. 
\end{equation}
Assuming GRH, we have
\[ \mc A_{d,a}(x)= \frac{\alpha_{d,a}}{\varphi(d)}
\frac{x}{\log x}+O_d\left( \frac{x\log\log x}{\log^2 x} \right). \]
\end{Thm}
The unconditional upper bound 
\eqref{eq:upperap} is not given by Moree in either
\cite{MOa}  or \cite{Mor2}, but is totally standard and
for a related problem worked out in detail in Hu et al.\,\cite{HKMS}.
\par We note for future use that if $d$ is odd, then
\begin{equation}
    \label{Rtwo}
    R(d,a)=
\begin{cases}
R(2d,a) & \text{~if~}2\nmid a;\\
R(2d,a+d) & \text{~otherwise}.
\end{cases}
\end{equation}
\begin{Rem}
\label{triviaaldichtheidnul}
Let $\Delta$ denote the discriminant of $\mathbb Q(\sqrt{\ell})$.
In case $\Delta\mid d$ and $\kronecker{\ell}{a}=1$, then using
quadratic reciprocity it is easy to see that $\mc A_{d,a}$ is empty and
so unconditionally $\alpha_{d,a}=0$. By Theorem 
\ref{thm-pietersap}, under GRH, $\alpha_{d,a}=0$ if
and only if $\Delta\mid d$ and $\kronecker{\ell}{a}=1$. 
\end{Rem}
\subsection{Near-primitive roots} 
\label{sec:npr}
Given integers $t\geq1$ and $g$, we set
\[ \mc P(g,t):=\{p:p\equiv1\bmod t, \quad {\rm ord}_p(g)=(p-1)/t  \}.   \]
The primes in $\mc P(g,t)$ are called \emph{near-primitive roots}.
Let $\epsilon>0$ be fixed. By \cite[Theorem 3.1]{HKMS} we have 
\begin{equation}
    \label{eq:upperbound}
\mc P(g,t)(x)\leq(\delta(g,t)+\epsilon)\frac{x}{\log x},
\end{equation}
for every $x$ sufficiently large, where 
\[ \delta(g,t)=\sum_{n=1}^\infty\frac{\mu(n)}{[\Q(\zeta_{nt},g^{1/nt}):\Q]}. \]
Assuming the Riemann Hypothesis for all number fields $\Q(\zeta_{nt},g^{1/nt})$ with $n$ squarefree, we have the sharper estimate
\begin{equation}
    \label{eq:sharper}
     \mc P(g,t)(x)=\delta(g,t)\frac{x}{\log x}+O_{g,t}\left( \frac{x\log\log x}{\log^2x} \right). 
\end{equation}     
Thus, conditionally, the set of primes $\mc P(g,t)$ has natural density $\delta(g,t)$.
This quantity was explicitly computed 
for $t=1$ by Hooley \cite{Hooley1} and for general $t$ by
Moree \cite{near}. It 
always equals a rational number times the
Artin constant, where the rational number may depend on both $g$ and $t$.
\subsection{Divisors of second order recurrences}
\label{sec:secondorder}
Let $\alpha$ and $\beta$ be integers.
The set of prime divisors $\mathcal Q_{\alpha,\beta}$ of the sequence 
$\{\alpha^n+\beta^n\}_{n=1}^{\infty}$ has been well studied. 
A prime $p\nmid \alpha \beta$ divides it if and only if 
$\ord_p(\alpha/\beta)$ is even (Moree \cite[Prop.\,2]{Mordivisors}).
Let $\mathcal Q_{\alpha,\beta}(a,d)$ be the set of primes $p\equiv a\bmod{d}$
in $\mathcal Q_{\alpha,\beta}$.
Moree and Sury \cite{MS} showed that in case $\alpha/\beta>0$, asymptotically,
\begin{equation}
\label{eq:sequencedivisor}
\mathcal Q_{\alpha,\beta}(a,d)(x) =\rho_{\alpha,\beta}(a,d)\frac{x}{\log x}+O\Big(\frac{x \log \log x}{\log^{7/6} x}\Big),\quad x\rightarrow \infty,
\end{equation}
where $\rho_{\alpha,\beta}(c,d)$ is an explicitly computable rational number.
For an informal proof in case $a=d=1$ and other references 
see \cite[Sect.\,9.2]{Moree}.

\subsection{Wieferich sets}\label{sec:wieferich}
Recall the definitions \eqref{wieferich1} and \eqref{wieferich2} of the
Wieferich sets.
The first case of Fermat's Last Theorem (FLTI) is the statement that, for any
 odd prime $p$, the equation $x^p +y^p = z^p$ does not have 
 positive integer solutions where
 none of $x, y, z$ is divisible by $p$. The \emph{generalized Wieferich criterion} 
 (for given $q$) is the statement that if FLTI fails for some prime $p$, then 
 $p\in {\mathcal W}_q$. This criterion has been proved 
 by Granville and Monagan \cite{GM} for all 
 $q\in \{2, 3, 5, 7,\ldots, 89\}$,
 the first 24 primes, continuing work by a great number of mathematicians,
 starting with Wieferich ($q=2$) and Mirimanoff ($q=3$). For $q=2$
 the only Wieferich numbers known below $10^{17}$ are $1093$ and
 $3511$. For more information see \url{https://en.wikipedia.org/wiki/Wieferich_prime}.
 It is believed that ${\mathcal W}_{\ell}(x)=O(\log \log x)$, but it is not even known
 whether
 ${\mathcal W}_{\ell}(x)=o(x/\log x)$ or not.
 The same is expected for
 ${\mathcal W}^{\varepsilon}_{\ell}(x)$.

\section{Counting $H^-$- and $H^+$-irregular primes}
\label{sec:hplusmincounting}
In Section \ref{hahaha} we will count $H^-$- and $H^+$-irregular primes in
prescribed arithmetic progression, here as a warm-up we consider
the problem of counting \emph{all} such primes.
\par Denote by $\mc P_{H^{\varepsilon}}$ the set of $H^{\varepsilon}$-irregular primes. By Lemma \ref{lem-H},  \eqref{eq:upperbound} and \eqref{eq:sequencedivisor}, we obtain unconditionally 
\begin{equation}
\label{Hmin}  
\mc P_{H^-}(x)\geq(1-\delta(\ell,1)-\delta(\ell,2)-\epsilon)\frac{x}{\log x} 
\end{equation}
and
\begin{equation}
\label{Hplus} 
\mc P_{H^+}(x)\geq(\rho_{\ell,1}(1,1)-\delta(\ell,1)-\epsilon)\frac{x}{\log x},
\end{equation}
where $\epsilon>0$ is arbitrary. 
\begin{Con}
\label{con:doubledensity}
If we require the primes counted by  $\mc P(g,t)$ and 
${\mathcal Q}_{\alpha,\beta}$ to be also $B$-regular, then the estimates \eqref{eq:sharper} and \eqref{eq:sequencedivisor} hold with 
$\delta(g,t)$ and $\rho_{\alpha,\beta}(a,d)$ replaced by $\delta(g,t)/\sqrt{e}$, respectively $\rho_{\alpha,\beta}(a,d)/\sqrt{e}$.
\end{Con}
This conjecture leads to the conjectures that
\begin{equation}
\label{Hmincon}  
\mc P_{H^-}(x)\sim \Big(1-\frac{1}{\sqrt{e}}(\delta(\ell,1)+\delta(\ell,2)\Big)\frac{x}{\log x} 
\end{equation}
and
\begin{equation}
\label{Hpluscon} 
\mc P_{H^+}(x)\sim\Big(1-\frac{1}{\sqrt{e}}(1-\rho_{\ell,1}(1,1)+\delta(\ell,1)\Big)\frac{x}{\log x}.
\end{equation}
For reasons of space we abstain from writing out 
\eqref{Hmin}, \eqref{Hplus}, \eqref{Hmincon}, and \eqref{Hpluscon} explicitly, but this can be done easily
using the following lemma.
\begin{Lem}\label{lem-dens}
Let $A$ be the Artin constant defined in \eqref{Artinconstantdef}. 
If $\ell=2$, then $\delta(2,1)=A$ and $\delta(2,2)=3A/4$. 
If $\ell\equiv1\bmod4$, then
\[ \delta(\ell,1)=A\left(1+\frac{1}{\ell^2-\ell-1}\right) \qquad \text{ and } \qquad \delta(\ell,2)=\frac{3A}{4} \left(1-\frac{1}{\ell^2-\ell-1}\right). \]
If $\ell\equiv3\bmod4$, then
\[ \delta(\ell,1)=A \qquad \text{ and } \qquad \delta(\ell,2)=\frac{3A}{4}\left(1+ \frac{1}{3(\ell^2-\ell-1)}\right). \]
We have
$\rho_{2,1}(1,1)=17/24$ and 
$\rho_{\ell,1}(1,1)=2/3$ for every odd prime $\ell$.
\end{Lem}
The quantities $\delta(\ell,1)$ and $\delta(\ell,2)$ are, for example, computed in Moree \cite{near}. The results for $\rho_{\ell,1}(1,1)$ were proved by Hasse \cite{Hasse} for Dirichlet density
and by Odoni \cite{Odoni} for natural density (which is what we use
here).
\section{Counting $\ell$-Genocchi irregular primes}
Recall that by Proposition \ref{prop:HvG}, counting $\ell$-Genocchi irregular primes is, except possibly for $p=\ell$, the same as counting
$H$-irregular primes. Corollary \ref{cor:simple} implies that
\[ \mc P_G(x)\geq \Big(\frac{1}{4}-\epsilon\Big)\frac{x}{\log x}. \]
By Proposition \ref{lem_Greg}, given $\epsilon>0$ arbitrary and fixed, we have for every $x$ sufficiently large
\[ \mc P_G(x)\geq \left( 1-\delta(\ell^2,2)-\epsilon \right)\frac{x}{\log x}, \]
Assuming Siegel's heuristic we conjecture that
\begin{equation}\label{eq:Gconj}
    \mc P_G(x)\sim \left( 1-\frac{\delta(\ell^2,2)}{\sqrt{e}} \right)\frac{x}{\log x}.
\end{equation}
By \cite[Theorem 1.10]{HKMS} for $\ell=2$, and 
the results of  \cite{near} for $\ell$ odd, 
we obtain
\begin{equation}
    \label{eq:double}
 \mc P_G(x)\geq \left( 1-\frac{3}{2}A-\epsilon \right)\frac{x}{\log x} \ \text{~and~} \
\mc P_G(x)\geq \left( 1-\frac{3A}{2}\left( 1+\frac{1}{3(\ell^2-\ell-1)}\right)-\epsilon \right)\frac{x}{\log x}, \end{equation}
respectively, where $\epsilon>0$ is arbitrary and fixed 
and $x$ sufficiently large.
For $\epsilon$ small enough, the constants involved are
in the interval $(0.4,0.44)$.
We conjecture that 
\[ \mc P_G(x)\sim \left(1-\frac{3A}{2\sqrt{e}} \right)\frac{x}{\log x}~(\ell=2) \
\text{~and~}  \ \mc P_G(x)\sim \left(1-\frac{3A}{2\sqrt{e}}\left( 1+\frac{1}{3(\ell^2-\ell-1)} \right) \right)\frac{x}{\log x}~(\ell>2), \]
with now the constants involved being in $(0.637,0.66)$.
See Table \ref{tabG1} for some numerical examples supporting these conjectures.

\section{Irregular primes in arithmetic progression}
The goal of this section is proving the main result of this paper, namely Theorem \ref{thm:GinAP}.
For $\ell=2$ this was already considered in \cite[Sect.\ 1.3.1]{HKMS}. Further, we study the extremal behavior of $\delta_{d,a}$.
\subsection{The $\ell$-Genocchi case}
Let $\ell$ be an odd prime number and $1\leq a < d$ be coprime integers.  In this section we consider the  set of rational (odd) primes
\begin{equation}
    \label{eq:pdadef}
\mc P_{d,a}:=\{p\equiv a\bmod d:\ord_p(\ell^2)=(p-1)/2  \}.  
\end{equation}
By Proposition \ref{lem_Greg} the primes 
in $\mc P_{d,a}$ are irregular.
Under GRH we have (see \cite[Theorem 3.1]{HKMS})
$$\lim_{x\rightarrow \infty}\frac{\mc P_{d,a}(x)}{\pi(x;d,a)}=\delta_{d,a},$$
with 
\begin{equation}
    \label{deltada}
\delta_{d,a} = \sum_{n=1}^\infty\frac{\varphi(d)\mu(n)c_a(n)}{[\Q(\zeta_{[d,2n]},\ell^{1/n}):\Q]},
\end{equation}
where $c_a(n)=1$ if the automorphism $\sigma_a$ of $\Q(\zeta_d)$ determined by $\sigma_a(\zeta_d)=\zeta_d^a$ is the identity on the field $\Q(\zeta_d)\cap\Q(\zeta_{2n},\ell^{1/n})$, and $c_a(n)=0$ otherwise.
\par We will now express $\delta_{d,a}$ as an Euler product in two
different ways: one more starting from first principles and another 
shorter one 
relying more heavily on
existing results.
\begin{Thm}\label{thm-ap}
Let $a$ and $d$ be coprime integers, $\mc P_{d,a}$ and $\delta_{d,a}$ as in 
    \eqref{eq:pdadef}, respectively \eqref{deltada}, 
and $\epsilon>0$ be arbitrary and fixed. Then for every $x$ sufficiently large we have
\begin{equation}
    \label{eq:upper}
\mc P_{d,a}(x)\leq (\delta_{d,a}+\epsilon)\frac{x}{\varphi(d)\log x}, 
\end{equation}
where
\[ \delta_{d,a}=A\,c(d,a)\,R(d,a), \]
with $R(d,a)$ as in \eqref{eq:rda}
and
\[ c(d,a) = \begin{cases} 
\frac{1}{2}\left( 3+\frac{1}{\ell^2-\ell-1} \right) & \text{if } \ell\nmid d, 4\nmid d;\\
 1+\frac{1}{\ell^2-\ell-1}  & \text{if } \ell\nmid d, 4\mid d, a\equiv 1\bmod 4;\\
 1 & \text{if } \ell\mid d, 4\nmid d, \left( \frac{a}{\ell} \right)=1;\\
2 & \text{if } 4\mid d \text{ and } a\equiv 3\bmod 4, \text{ or } \ell\mid d \text{ and } \left( \frac{a}{\ell} \right)=-1;\\
0 & \text{if } 4\ell\mid d , \left( \frac{a}{\ell} \right)=1 , a\equiv1\bmod4 .
\end{cases} \]
Assuming GRH we have
\[ \mc P_{d,a}(x)= \frac{\delta_{d,a}}{\varphi(d)}\frac{x}{\log x}+O_d\left( \frac{x\log\log x}{\log^2 x} \right). \]
\end{Thm}
\begin{Rem}
\label{rem:alternative}
Note that alternatively we can write
\[ c(d,a) = \begin{cases} 
\frac{1}{2}\left( 3+\frac{1}{\ell^2-\ell-1} \right) & \text{if }  4\nmid d, \ell\nmid d;\\
\frac{1}{2}( 3-\kronecker{a}{\ell}) & \text{if }  4\nmid d, \ell\mid d;\\
  1+\frac{1}{\ell^2-\ell-1}   & \text{if } 4\mid d, a\equiv 1\bmod 4, \ell\nmid d;\\
  1-\kronecker{a}{\ell}  & \text{if } 4\mid d, a\equiv 1\bmod 4, \ell\mid d;\\
2 & \text{if } 4\mid d \text{ and } a\equiv 3\bmod 4.\\
\end{cases} \]
\end{Rem} 
\begin{Rem}
{}From Theorem \ref{thm-ap} we infer that the analogue of identity
\eqref{Rtwo} holds for
$\delta_{d,a}$ as well, which is consistent with the fact that the analogue
of this identity also holds for $\mc P_{d,a}$.
\end{Rem}

\subsubsection{The proof of Theorem \ref{thm-ap}}
The proof requires a few preliminary lemmas. The first is merely a special
case of Lemma 3.1 of \cite{Mor2}.
\begin{Lem}
\label{lem:Ssums}
Put \[ \omega_d(n):=\frac{n\varphi([d,n])}{\varphi(d)}. \]
It is a multiplicative function in $n$.
For $m\geq1$, let
\[ 
S(m)=\sum_{\substack{n\geq1,\ m\mid n \\ a\equiv 1\bmod (d,n)}} \frac{\mu(n)}{\omega_d(n)}, \qquad
S_2(m)=\sum_{\substack{n\geq1,\ [2,m]\mid n \\ a\equiv 1\bmod (d,n)}} \frac{\mu(n)}{\omega_d(n)}.
\]
We have $S_2(m)=-S(m)$. Further, $S(1)=AR(d,a)$
and 
$$S(\ell)=
\begin{cases}
-\frac{A}{\ell^2-\ell-1}R(d,a) & \text{ if } \ell\nmid d;\\ 
-\frac{A}{\ell-1}R(d,a) & \text{ if } \ell\mid d,
\end{cases}
$$
where $\ell$ is an odd prime.
\end{Lem}
We recall the law  of quadratic reciprocity for the Jacobi
symbol: if $j$ and $k$ are odd coprime positive integers, then
\begin{equation}
\label{Jacobi}    
\kronecker{j}{k}\kronecker{k}{j}=(-1)^{(j-1)(k-1)/4}.
\end{equation}

In the following we set  $\ell^*=(-1)^{(\ell-1)/2}\ell$.

\begin{Lem}
\label{lem:actiononsquare}
Let $\sigma_a$ be the automorphism of $\Q(\zeta_d)$ uniquely determined by $\sigma_a(\zeta_d)=\zeta_d^a$ with $a$ a positive integer coprime to $d$. 
If $4\ell\mid d$, then
$\sigma_a(\sqrt{\ell})=\kronecker{\ell}{a}\sqrt{\ell}$.
If $\ell\mid d$, then
$\sigma_a(\sqrt{\ell^*})=\kronecker{a}{\ell}\sqrt{\ell^*}$.
\end{Lem}
\begin{proof}
The quadratic Gauss sum expresses $\sqrt{\ell^*}$ as an element in $\Q(\zeta_d)$, which allows
one to determine $\sigma_a(\sqrt{\ell^*})$ and from this, using $\sigma_a(i)=i^a$, also $\sigma_a(\sqrt{\ell})$.
Invoking \eqref{Jacobi} we can formulate the outcome in a more compact way.
%This easily follows from the evaluation of Gauss of the quadratic Gauss sum and \eqref{Jacobi}.
\end{proof}
\begin{Rem}
This can also be proved using that $\sigma_a(\sqrt{\ell})/\sqrt{\ell}$ is a character, 
see \cite[Lemma 2.1]{Mor2}.
\end{Rem}
\begin{Lem}\label{lem-coeff}
Let $n$ be a squarefree integer. If $n$ is odd, or $\ell\mid n$, or $n$ is even and $\ell\nmid nd$, then we have: $c_a(n)=1$  if and only if $a\equiv1\bmod(d,2n)$. If $n$ is even, $\ell\nmid n$ and $\ell\mid d$, then we have: $c_a(n)=1$  if and only if $a\equiv1\bmod(d,2n)$ and $\left(\frac{a}{\ell}\right)=1$.
\end{Lem}

\begin{proof}
Let us set $I:=\Q(\zeta_d)\cap\Q(\zeta_{2n},\ell^{1/n})$. By Kummer theory we may argue that, since $\Q(\zeta_\infty)\cap\Q(\zeta_{2n},\ell^{1/n})$ is a finite abelian extension of $\Q(\zeta_{2n})$, it is of the form $\Q(\zeta_{2n},\ell^{e/n})$ for some $e\geq0$, and hence by Schinzel's Theorem \cite[Theorem 2]{schinzel} it is either $\Q(\zeta_{2n})$ if $n$ is odd, or $\Q(\zeta_{2n},\sqrt{\ell})$ if $n$ is even. 
The latter extension equals $\Q(\zeta_{2n})$ if $\ell\mid n$ (as $\Q(\sqrt{\ell})\subseteq \Q(\zeta_{4\ell})$ and we already have that $n$ is even).

Thus, if $n$ is odd or $\ell\mid n$, we deduce that $I=\Q(\zeta_{(d,2n)})$.
For $n$ even with $\ell\nmid n$, it suffices to notice that $\Q(\sqrt{\ell^*})$ is contained in $\Q(\zeta_d)$ if and only if $\ell\mid d$. If $\ell\nmid d$, then we have $I=\Q(\zeta_{(d,2n)})$. If $\ell\mid d$, then noticing that $\Q(\zeta_{(d,2n)}) \subseteq I \subseteq \Q(\zeta_{(d,2n)},\zeta_\ell)$, we deduce that $I=\Q(\zeta_{(d,2n)},\sqrt{\ell^*})$  (where $\Q(\sqrt{\ell^*})$ is not contained in $\Q(\zeta_{(d,2n)})$).

If $I=\Q(\zeta_{(d,2n)})$, then the automorphism $\sigma_a$ fixes $I$ if and only if $a\equiv1\bmod(d,2n)$. Suppose now that $I=\Q(\zeta_{(d,2n)},\sqrt{\ell^*})$ and $\Q(\sqrt{\ell^*})\not\subseteq\Q(\zeta_{(d,2n)})$. 
By Lemma \ref{lem:actiononsquare} the automorphism $\sigma_a$ fixes  $\Q(\sqrt{\ell^*})\subseteq\Q(\zeta_\ell)$ if and only if $a$ is a square modulo $\ell$, i.e.\  $\left(\frac{a}{\ell}\right)=1$. 
%This is because the former holds if and only if $\tau=\sigma_a|_{\Q(\zeta_\ell)}$ fixes  $\Q(\sqrt{\ell^*})$, which is equivalent to $\tau$ being in the only subgroup of $\Gal(\Q(\zeta_\ell)/\Q)\cong(\Z/\ell\Z)^\times$ which has index $2$. This subgroup corresponds to the subgroup of squares modulo $\ell$.
%Therefore,  $\sigma_a$ fixes $I$ if and only if it fixes both subfields $\Q(\zeta_{(d,2n)})$ and $\Q(\sqrt{\ell^*})$, which is equivalent to $a\equiv1\bmod(d,2n)$ and $\left(\frac{a}{\ell}\right)=1$.
\end{proof}

\begin{Rem}\label{rem-details}
Let $n$ be a squarefree number. We collect here some technical details on the numbers $(d,2n)$ and $[d,2n]$, and on the condition $a\equiv 1\bmod (d,2n)$.  
\begin{itemize}
\item If $d$ is odd, or $4\nmid d$ and $n$ is even, then $(d,2n)=(d,n)$ and $[d,2n]=2[d,n]$. 
\item If $d$ is even and  $n$ is odd, or $4\mid d$, then $(d,2n)=2(d,n)$ and $[d,2n]=[d,n]$. 
\end{itemize}
If $d$ is even and $n$ is odd, then we have  $a\equiv1\bmod 2(d,n)$ if and only if $a\equiv1\bmod (d,n)$, because $a$ must be odd. If $4\mid d$ and $n$ is even, then $a\equiv1\bmod 2(d,n)$ holds only if $a\equiv 1\bmod4$, and in this case $a\equiv1\bmod 2(d,n)$ is equivalent to $a\equiv1\bmod (d,n)$.
\end{Rem}

\begin{proof}[Proof of Theorem \ref{thm-ap}]
Recall that the degree $[\Q(\zeta_{[d,2n]},\ell^{1/n}):\Q]$ equals $\varphi([d,2n])n/2$ if $n$ is even and $\ell\mid[d,n]$, and it equals $\varphi([d,2n])n$ otherwise.
For the computation of the density $\delta_{d,a}$ we distinguish the cases $\ell\nmid d$ and $\ell\mid d$.

\emph{Case 1: $\ell\nmid d$.}
Using Lemma \ref{lem-coeff} we see that  the expression \eqref{deltada} simplifies to  
\[  \delta_{d,a}=\sum_{\substack{n\ge 1\\ a\equiv 1\bmod (d,2n)}}\frac{\varphi(d)\mu(n)}{[\Q(\zeta_{[d,2n]},\ell^{1/n}):\Q]}. \]
In view of the degree formulas, we obtain
\begin{equation}\label{eq1}
   \delta_{d,a}  =
\Bigg(\sum_{\substack{n\geq1\\ a\equiv 1\bmod (d,2n)}}+
		\sum_{\substack{2\ell\mid n \\ a\equiv 1\bmod (d,2n)}}
\Bigg) \frac{\mu(n)}{\omega_d(n)}, 
\end{equation}
with $\omega_d(n)$ as in Lemma \ref{lem:Ssums}.

\emph{Case 1.1: $4\nmid d$.} By Remark \ref{rem-details}, from \eqref{eq1} we have
\begin{align*}
\delta_{d,a} & =  \Bigg(
		\sum_{2\nmid n  \atop a\equiv 1\bmod (d,n)}  +
		\frac{1}{2}\sum_{2\mid n  \atop a\equiv 1\bmod (d,n)} +
		\frac{1}{2}\sum_{2\ell\mid n  \atop a\equiv 1\bmod (d,n)} \Bigg)\frac{\mu(n)}{\omega_d(n)} \\
		& =  \Bigg(
		\sum_{n\geq1  \atop a\equiv 1\bmod (d,n)}  -
		\frac{1}{2}\sum_{2\mid n  \atop a\equiv 1\bmod (d,n)} +
		\frac{1}{2}\sum_{2\ell\mid n  \atop a\equiv 1\bmod (d,n)} \Bigg)\frac{\mu(n)}{\omega_d(n)}.
\end{align*}
%Now and in the following cases we apply \cite[Lemma 3.1]{Mor2}, keeping the same notation as there. Notice that $\ell$ divides  $n$ if and only if $\ell$ divides $[d,n]$, as $\ell\nmid d$.  
Then using Lemma \ref{lem:Ssums} we obtain
%\begin{align*}
$$\delta_{d,a}   = S(1)-\frac{1}{2}S_2(1)+\frac{1}{2}S_2(\ell)=\frac{1}{2}\left(3S(1)-S(\ell) \right) 
= \frac{A}{2} \left(3 +\frac{1}{\ell^2-\ell-1}\right) R(d,a).$$

%\end{align*}

\emph{Case 1.2: $4\mid d$.} In view of Remark \ref{rem-details}, if $a\equiv1\bmod4$, then \eqref{eq1} becomes
$$
\delta_{d,a} =   \Bigg(
		\sum_{n\geq1  \atop a\equiv 1\bmod (d,n)}  +
		\sum_{2\ell\mid n  \atop a\equiv 1\bmod (d,n)} \Bigg)\frac{\mu(n)}{\omega_d(n)} 
		= S(1)+S_2(\ell) 	= A\left( 1+\frac{1}{\ell^2-\ell-1}\right)R(d,a).
$$
If $a\equiv3\bmod4$, then from \eqref{eq1} 
we are left with
\begin{equation}
\label{easycase}
    \delta_{d,a}= \sum_{2\nmid n  \atop a\equiv 1\bmod (d,2n)}\frac{\mu(n)}{\omega_d(n)}
=  2S(1) = 2AR(d,a).
\end{equation}

\emph{Case 2: $\ell\mid d$.} 
We distinguish the two cases: $a$ is a square modulo $\ell$ or not.

\emph{Case 2.1: $\left(\frac{a}{\ell}\right)=-1$.} 
Notice that the condition $a\equiv 1\bmod (d,2n)$ implies in particular that $\ell\nmid n$, otherwise we would have $a\equiv1\bmod\ell$ and hence a contradiction with the assumption. Thus, by Lemma \ref{lem-coeff} and Remark \ref{rem-details} we have
$$
\delta_{d,a}  = \sum_{2\nmid n  \atop a\equiv 1\bmod (d,n)}\frac{\mu(n)}{\omega_d(n)} \\
= S(1)-S_2(1) = 2A R(d,a).
$$

\emph{Case 2.2: $\left(\frac{a}{\ell}\right)=1$.} 
By Lemma \ref{lem-coeff} we obtain
\[ \delta_{d,a}  =\sum_{n\geq1 \atop a\equiv 1\bmod (d,2n)}
\frac{\varphi(d)\mu(n)}{[\Q(\zeta_{[d,2n]},\ell^{1/n}):\Q]}
=\Bigg(\sum_{2\nmid n \atop a\equiv 1\bmod (d,2n)}
+2\sum_{ 2\mid n \atop a\equiv 1\bmod (d,2n)}\Bigg)\frac{\mu(n)}{\omega_d(n)}. \]
In the following we take  Remark \ref{rem-details} into account.
If $4\nmid d$, then 
\[ \delta_{d,a} = S(1) = A R(d,a). \]
%\begin{align*}
%\delta_{d,a}  & = \frac{1}{\varphi(d)}( S(1)+\frac{1}{2}S_2(1) ) = \frac{S(1)}{2\varphi(d)} \\
% & = \frac{A}{2\varphi(d)}\prod_{p\mid(a-1,d)} \left( 1-\frac{1}{p} \right)\prod_{p\mid d}\left(1+\frac{1}{p^2-p-1}\right).
%\end{align*}
If $4\mid d$ and $a\equiv 1\bmod 4$, then
\[ \delta_{d,a} = S(1)+S_2(1)=0. \]
If $4\mid d$ and $a\equiv 3\bmod 4$, then
\[ \delta_{d,a} = S(1)-S_2(1)=2 A R(d,a).  \qedhere\]
\end{proof}

\subsubsection{Alternative proof of Theorem \ref{thm-ap}}\label{alternative}
\begin{proof}[Proof of Theorem \ref{thm-ap}]
We 
start by noting, 
cf.\,\eqref{eq:trivial}, that
\begin{equation}
    \label{altset}
\mc P_{d,a}=\{p:p\equiv a\bmod d, \, \ord_p(\ell)=p-1\text{~or~}p\equiv 3\bmod 4\text{~and~}\ord_p(\ell)=(p-1)/2\}.
\end{equation}
Without loss of generalization we may assume that $4$ divides $d$: if $4\nmid d$ we split the progression into two, according to whether $a\equiv 1\bmod 4$   or  $a\equiv 3\bmod 4$, and add
the results.
Thus, if $a\equiv 1\bmod 4$, then we just
have
\[ \mc P_{d,a}=\{p:p\equiv a\bmod d, \, \ord_p(\ell)=p-1  \}.\]
Using Theorem \ref{thm-pietersap} and the law of quadratic 
reciprocity we conclude that
\begin{equation}
\label{specialcase}
\delta_{d,a}=A R(d,a)c(d,a)
\text{~with~}
 c(d,a) = \begin{cases} 
1-\left( \frac{a}{\ell} \right) & \text{if } a\equiv 1\bmod 4,\, \ell\mid d;\\
 1+\frac{1}{\ell^2-\ell-1}  & \text{if } a\equiv 1\bmod 4,\,\ell\nmid d.\\
\end{cases} 
\end{equation}
If $a\equiv 3\bmod 4$, then  by \eqref{easycase} we
arrive at
\[ \delta_{d,a}=AR(d,a)c(d,a), \text{~with~}c(d,a)=2.\]
Now let us suppose that $4\nmid d$. 
We put $d_1=\lcm(4,d)$.
We let $a_1$ and $a_3$ be integers such that
$a_j\equiv a\bmod d$ and
$a_j\equiv j\bmod 4$. Noting that
$R(d_1,a_j)=R(d,a)$ and $\varphi(d_1)=2\varphi(d)$, we conclude
that
$$\delta_{d,a}=\frac{\delta_{d_1,a_1}+\delta_{d_1,a_3}}{2}=\frac{A}{2}(c(d_1,a_1)+c(d_1,a_3))R(d,a).$$
We find that $c(d_1,a_3)=2$ and noticing that if $\ell\mid d$, then we have
$\kronecker{a_1}{\ell}=\kronecker{a}{\ell}$, 
we obtain from \eqref{specialcase} that 
$$c(d_1,a_1)=\begin{cases} 
1-\left( \frac{a}{\ell} \right) & \text{if }\ell\mid d;\\
 1+\frac{1}{\ell^2-\ell-1}  & \text{if }\ell\nmid d.
\end{cases}
$$
The proof (with the reformulation of $c(d,a)$ as 
given in Remark \ref{rem:alternative}) is now easily completed.
\end{proof}
\begin{Example}
Take $d=4$ and $\ell=3$. 
Moree and Zumalac\'arregui \cite{MZ} 
crucially made use of the sets $\mc P_{4,1}$ and $\mc P_{4,3}$ in 
their solution of a conjecture of Salajan.
By a simple direct computation 
they showed that
$\delta_{4,1}=6A/5$ and $\delta_{4,3}=2A$ \cite[Appendix A]{MZ}, in
agreement with our results. These sets play also an
important role in the resolution of Browkin's generalization of the Salajan conjecture 
by Ciolan and Moree \cite{CM}.
\end{Example}
\subsubsection{The extremal behavior of $\delta_{d,a}$}\label{sec:extremal}
Theorem \ref{thm:GinAP} gives a lower bound for $\mc P_G(d,a)(x)$. In this section we 
will study how small and large this lower bound can be. For $\ell=2$ this was done in \cite[Sect.\,2.2]{HKMS}. This amounts to bounding $\delta_{d,a}$, 
which a priori satisfies $0\le \delta_{d,a}\le 1$, as it is a relative density.
%\section{Todo}
%\noindent
%-Generating series of $H^+$ and $H^-$?
\par \textbf{Small $\delta_{d,a}$}. We put
$$
G(d)=A\prod_{p \mid d}\left(1+\frac{1}{p^2-p-1}\right).$$
Note that
$$
G(d)=\prod_{p \nmid d}\left(1-\frac{1}{p(p-1)}\right)<1.
%=1+O(\frac{1}{q}),
$$
Recall that
\begin{equation*}
AR(d,a)=G(d)\prod_{p \mid (a-1,d)}
\left(1-\frac{1}{p}\right).
\end{equation*}
If $d$ is even, then $(a-1,d)$ is even and we infer that
$AR(d,a)<1/2$. If $d$ is odd, then $G(d)<1/2$ and 
again $AR(d,a)< 1/2$. As $c(d,a)\le 2$, we infer that
\[    
\delta_{d,a}=AR(d,a)c(d,a)< 1.
\]
Clearly 
$$AR(d,a)\ge AR(d,1)=\frac{\varphi(d)}{d}G(d).$$
The ratio $\varphi(d)/d$ takes on local minima on
products 
of consecutive primes. For these products we see that
$G(d)$ tends to 1, on noting that 
\begin{equation}
\label{eq:tail}
\prod_{p\ge q}\Big(1-\frac{1}{q(q-1)}\Big)=1+O\Big(\frac{1}{q}\Big).
\end{equation}
Using this and Mertens' theorem 
(see \cite[Theorem 13.13]{Apostol}), with $\gamma$ Euler's constant,
\begin{equation}
    \label{mertens}
    \prod_{p\le x}\Big(1-\frac{1}{p}\Big)\sim \frac{e^{-\gamma}}{\log x},
\end{equation}
we can then infer that
\begin{equation}
\label{eq:liminf}    
\liminf_{d \to \infty} AR(d,1)\log\log d = e^{-\gamma}. 
\end{equation}
The argument is similar to that of the proof of the classical 
result (see, for instance, \cite[Theorem 13.14]{Apostol}) $$
\liminf_{d \to \infty} \frac{\varphi(d)}{d}  \log\log d = e^{-\gamma}. 
$$

\begin{Prop}  \label{inf}
Let $\ell$ be a prime.
We have 
$$
\liminf_{d\to \infty} \min_{\substack{1\le a < d \\ (a,d)=1\\ \delta_{d,a}>0}} \delta_{d,a}\log\log d  = e^{-\gamma}.
$$
\end{Prop}
\begin{proof}
For $\ell=2$ this result is Proposition 2.3 of 
\cite{HKMS}.
Our proof for odd $\ell$ is in the same spirit.
\par If $\delta_{d,a}>0$, then $c(d,a)\ge 1$.
Hence $\delta_{d,a}\log \log d\ge AR(d,a)\log \log d\ge AR(d,1)\log \log d$.
It now follows from \eqref{eq:liminf} that
the limes inferior 
is $\ge e^{-\gamma}$. 
In order to show that this bound is actually sharp we 
will show that
$\lim_{n\rightarrow \infty}\delta_{d_n,1}=e^{-\gamma}$, where 
$d_n=\prod_{3\le p\le n}p$.
\par Let $n\ge \ell$ be 
arbitrary. 
We have $c(d_n,1)=1$ and
$$
\delta_{d_n,1}=\Big(1+O\Big(\frac{1}{n}\Big)\Big)\prod_{2\le p\le n}\Big(1-\frac{1}{p}\Big),$$
where we used that
$$\prod_{p\nmid d_n}\Big(1-\frac{1}{p(p-1)}\Big)=\frac{1}{2}\prod_{p>n}\Big(1-\frac{1}{p(p-1)}\Big)=\frac{1}{2}+O\Big(\frac{1}{n}\Big),$$ with the last equality following from \eqref{eq:tail}.
Using Mertens' theorem \eqref{mertens}
and the prime number theorem in
the form $\log d_n\sim n$, we deduce that, as $n$ tends to infinity, 
$$
\delta_{d_n,1}\sim \frac{e^{-\gamma}}{\log n}\sim 
\frac{e^{-\gamma}}{\log \log d_n},$$
completing the proof.
 \end{proof}
%\pieter{This paragraph is in progress...}
%\begin{Lem} 
%\label{Efthymios}
%Let  $\ell$ be an odd prime and $\epsilon\in \{-1,1\}$. We have
%%$$\limsup_{n\to \infty} \max_{\substack{1\le a < d_n \\ (a,d_n)=1\\ a\equiv %1\bmod3}} 
%%\prod_{p\mid (a-1,d_n)}(1-\frac{1}{p})= \frac{2}{3}.$$ 
%$$\limsup_{n\to \infty} \max_{\substack{1\le a < d_n \\ (a,d_n)=1\\ %\kronecker{a}{\ell}=\epsilon}} \,\,
%\prod_{p\mid (a-1,d_n)}(1-\frac{1}{p})=
%\begin{cases}
%\frac{2}{3} & \text{if~}\epsilon=1\text{~and~}\ell=3;\\ 
%1 & \text{otherwise}.
%\end{cases}
%$$
%Furthermore, we have
%$$\limsup_{n\to \infty} \max_{\substack{1\le a < 4d_n \\ (a,4d_n)=1\\ a\equiv %1\bmod4}} \,\,
%\prod_{p\mid (a-1,4d_n)}(1-\frac{1}{p})= \frac{1}{2}.$$ 
%\end{Lem}
%\begin{proof}
%...apply Lemma \ref{lem:HB}....
%Lemma 1 of Heath-Brown \cite{HB} implies that 
%$$\#\{p\le x:p\equiv 5\bmod{8},\,\frac{p-1}{4}=P_2(\alpha)\}\gg %\frac{x}{\log^2x}.\qedhere$$
%\end{proof}
\par \textbf{Large $\delta_{d,a}$}. 
%The final assertion of Theorem \ref{thm:GinAP} 
%is a
%corollary of Proposition \ref{casesbounds}. 
Proposition \ref{casesbounds} gives information on how large $\delta_{d,a}$ can be.
In its proof we make
use of the elementary concepts of $a$- and $d$-sequence, which we now 
introduce.
Let $q_1,q_2,\ldots$ be a, possibly finite, sequence of pairwise coprime
 integers and $\alpha_1,\alpha_2,\ldots$ any integer sequence of equal length.
Put $d_k=\prod_{j=1}^k q_j$.
By the Chinese remainder theorem the system of congruences
\begin{equation}
 \label{eq:congsystem}
 x\equiv \alpha_1\bmod{q_1},\,x\equiv \alpha_2\bmod{q_2},\ldots,\,x\equiv \alpha_k\bmod{q_k},
 \end{equation}
 is equivalent with 
 $$x\equiv a_k\bmod{d_k},\quad 0\le a_k<d_k,$$
 with $a_k$ unique (which is a consequence of the
 Chinese remainder theorem). Thus to the system \eqref{eq:congsystem} we can associate 
 the $(a,d)$-sequence $(a_1,d_1),(a_2,d_2),\ldots$. 
% The sequence $\{a_1,a_2,\ldots\}$ we call the
% associated Chinese remainder sequence. 
For example, the system of congruences
 $x\equiv 3\bmod4$ and $x\equiv 2\bmod{p_i}$, with $p_i$ running
 through the consecutive odd primes, leads to  
the $a$-sequence $\{3,11,47,107,3467,45047,\ldots\}$.
%natural question to wonder how $\delta_{d,a}$ is
%bounded above according to which of the classes distinguished
%in Remark \ref{rem:alternative} it is in.
%We can refine this result further by restricting to certain
%subclasses of the pairs of couples $(a,d)$. We do this following
%the case distinctions made in Remark ...
\begin{Prop}
\label{casesbounds}
Let $\ell$ be an odd prime.
We have \[ \delta_{d,a} < \begin{cases} 
\frac{1}{4}\Big(3-\frac{2}{\ell(\ell-1)}\Big) & \text{if } \ell\nmid d, 4\nmid d;\\
\frac{1}{2}  & \text{if } \ell\nmid d, 4\mid d, a\equiv 1\bmod 4;\\
\frac{1}{3} & \text{if } 3\mid d, \ell=3, 4\nmid d, a\equiv 1\bmod3;\\
\frac{1}{2} & \text{if } \ell\mid d, \ell>3, 4\nmid d, \left( \frac{a}{\ell} \right)=1;\\
1 & \text{if } 4\mid d \text{ and } a\equiv 3\bmod 4;\\
1 & \text{if } \ell\mid d \text{ and } \left( \frac{a}{\ell} \right)=-1,
\end{cases} \]
and $\delta_{d,a}=0$ in the remaining cases.
All of the upper bounds are sharp in the sense that they do not
always hold if an arbitrary $\epsilon>0$ is subtracted from them.
\end{Prop}
\begin{proof}
Starting point is the formula
$$
\delta_{d,a}=c(d,a)\prod_{p \mid (a-1,d)}
\left(1-\frac{1}{p}\right)\prod_{p \nmid d}\left(1-\frac{1}{p(p-1)}\right)=c(d,a)\,\Pi_1\,\Pi_2,$$
say.
The six subcases we denote by respectively a,b,c,d,e and f. 
For each of them the conditions imposed on $a$ and $d$ ensure
that $(a-1,d)$ has certain prime factors, e.g., if $4\mid d$, then
$2\mid (a-1,d)$. These factors are indicated in the $\Pi_1$ column
of Table 1. Likewise certain factors have to appear in the $\Pi_2$ column.
An entry $e$ in the $\Pi_1$ column leads to a factor 
$1-\frac{1}{e}$ in $\delta_{d,a}$, in the $\Pi_2$ column $e$ leads to
$1-\frac{1}{e(e-1)}$. 
Further, in absence of an entry, we put a $1$ as factor.
Clearly multiplying everything and also multiplying
by $c(d,a)$ (which has a fixed value in each subcase), leads to an upper bound
for $\delta_{d,a}$. For example, in subcase a we obtain
%$$\lim_{n\rightarrow \infty}\delta_{d_n/\ell,2}=
$$\frac{1}{2}\Big(3+
\frac{1}{\ell^2-\ell-1}\Big)\cdot 1\cdot \frac{1}{2}\Big(1-\frac{1}{\ell(\ell-1)}\Big)=\frac{1}{4}\Big(3-\frac{2}{\ell(\ell-1)}\Big),$$
with the factor before the dot being $c(d,a)$, $1$ being the contribution
to $\Pi_1$ and the rest being the contribution to $\Pi_2$.
This explains the upper bound for $\delta_{d,a}$ in subcase a, and the other
upper bounds are read off similarly from Table 1.
\par  It remains to establish the sharpness of the upper bounds. We do
this by indicating six families $(a_j,d_j)$ such that 
$\delta_{d_j,a_j}$ tends to the indicated upper bound. 
The $a_j$ are solutions of a certain system of congruences, the
default system being
$$x\equiv 2\bmod{3},\,x\equiv 2\bmod{5},\,x\equiv 2\bmod{7},\ldots,\,x\equiv 2\bmod{11},\ldots$$
where the moduli run over the consecutive odd primes.
In each of the six cases we make some small
modifications to the default system involves at most 
the moduli $3$ and $\ell$, and we possibly add a congruence
modulo $4$.\\ 
\textbf{Construction of the $(a,d)$-sequences}\\
a) We remove the congruence $x\equiv 2\bmod \ell$.
Trivially now the $a$-sequence is $2,2,2,2,\ldots$.\\
b) We start with the congruence $x\equiv 1\bmod 4$ and
remove the congruence $x\equiv 2\bmod \ell$. Thus for
$\ell=5$ we find, for example, the $a$-sequence
$1,5,65,233,3005,51053,\ldots$.\\
c) We start with $x\equiv 1\bmod 3$ and obtain the 
$a$-sequence
$1,7,37,772,10012,85087\ldots$.\\
d) Here we need $\ell>3$. The congruence 
$x\equiv 2\bmod \ell$ is changed to 
$x\equiv 4\bmod \ell$. Thus for
$\ell=5$ we find, for example, the $a$-sequence
$2,14,44,464,12014,102104,\ldots$.\\
e) We start with the congruence $x\equiv 3\bmod 4$, 
leading to an $a$-sequence $3,11,47,107,3467,45047,\ldots$\\
f) We change $x\equiv 2\bmod \ell$ to 
$x\equiv n_0\bmod \ell$, with $n_0$ the smallest non-residue modulo $\ell$. For $\ell=7$ we obtain $a$-sequence
$2,2,17,332,10727,145862,\ldots$. (Note that we get the default congruence system 
if and only if $\ell\equiv \pm 3\bmod8$.)\\

\par Let $k\ge 2$. In each of the above subcases the
constructed $(a,d)$-sequence has the property
that $(a_k,d_k)=1$ and in addition $p\mid (a_k-1,d_k)$
if and only if $p$ is in the $\Pi_1$ column. If $k$ is large
enough, the primes that appear in the $\Pi_2$ column 
are precisely those indicated in that column, with 
in addition all prime $p\ge p_0$ for some $p_0$ tending
to infinity with $k$. We conclude that, as $k$ gets larger,
$$\delta_{a_k,d_k}=(\text{upper bound})\prod_{p\ge p_0}\left(1-\frac{1}{p(p-1)}\right)\rightarrow \text{upper bound},$$
concluding the proof.
\end{proof}

\centerline{\bf Table 1}
\begin{center}
\begin{tabular}{|c|c|c|c|}
\hline
subcase & $c(d,a)$ & $\Pi_1$ & $\Pi_2$\\
\hline
a & $\frac{1}{2}( 3+\frac{1}{\ell^2-\ell-1})$ &  & $2,\ell$\\
\hline
b & $1+\frac{1}{\ell^2-\ell-1}$ &  $2$ & $\ell$\\
\hline
c & $1$ & $3$ & $2$\\
\hline
d & $1$ &  & $2$\\
\hline
e & $2$ & $2$ & \\
\hline
f & $2$ & $2$ if $2\mid d$ & $2$ if $2\nmid d$\\
\hline
\end{tabular}
\end{center}
\vskip .4cm
\begin{Rem}
Our choice of the six sequences $(a_k,d_k)$ was very canonical, but in fact given
$d_k$ many choices of $a_k$ are allowed. E.g., in subcase c there are
$\varphi(d_k)\prod_{3\le p\le p_k}(p-2)/(p-1)$ choices allowed, with $p_k$ the
$k$th odd prime. This number is asymptotically equal to
$c_1\varphi(d_k)/\log \log d_k$, for some positive constant $c_1$. The same conclusion,
with possibly different $c_1$, is valid for the other five sequences.
\end{Rem}
It remains to deal with the case $\ell=2$. Proceeding as above one
deduces from Proposition 1.12 of \cite{HKMS} the following result.
\begin{Prop}
Let $\ell=2$. We have
$$
\delta_{d,a}=
\begin{cases}
\frac{3}{4}  & \text{if~}4\nmid d;\\
\frac{1}{2}  & \text{if~}4\mid d,8\nmid d,~a\equiv 1~({\rm mod~}4);\\
1  & \text{if~}4\mid d,8\nmid d,~a\equiv 3~({\rm mod~}4);\\
1  & \text{if~}8\mid d,~a\not\equiv 1~({\rm mod~}8);
\end{cases}
$$
and $\delta_{d,a}=0$ in the remaining cases.
All of the upper bounds are sharp in the sense that they do not
always hold if an arbitrary $\epsilon>0$ is subtracted from them.
\end{Prop}
\subsubsection{The proof of Theorem \ref{thm:GinAP}}\label{sec:thm3}
\begin{proof}
Let $\epsilon>0$ be arbitrary.
Recalling the definition \eqref{eq:pdadef} of $\mc P_{d,a}$, we find that, for all
$x$ sufficiently large,
$$\mc P_G(d,a)(x)\ge \pi(x;d,a)-\mc P_{d,a}(x)+O(1)\ge (1-\delta_{d,a}-\epsilon)\frac{x}{\varphi(d)\log x},$$
where we used the upper bound \eqref{eq:upper} for $\mc P_{d,a}(x)$, and the well-known asymptotic
$\pi(x;d,a)\sim x/(\varphi(d)\log x)$. The claims regarding 
$1-\delta_{d,a}$ are an immediate consequence of Proposition \ref{casesbounds}.
\end{proof}
\begin{Example}
Setting $a=d=1$ we have $\mc P_G(1,1)(x)=\mc P_G(x)$ in our earlier notation.
Theorem \ref{thm:GinAP} then yields the second inequality in \eqref{eq:double}.
\end{Example}
\section{Prime divisors of the $H^{\varepsilon}$-sequences}
\label{hahaha}
\subsection{The $H^-$-sequence} 
In this section we obtain the results on the growth
behavior of $\mc P_{H^{\varepsilon}}(d,a)(x)$ that are needed
in order to prove Theorem \ref{thm:HinAP}.
\subsubsection{Statement of results}
Let $\ell$ be an odd prime number and $1\leq a < d$ be coprime integers.  In this section we consider the  set of rational (odd) primes
\[
\mc P_{d,a}^-:=\{p\equiv a\bmod d:\ord_p(\ell)=(p-1)/2 \text{~or~} \ord_p(\ell)=p-1  \}.  \]
As Theorem \ref{thm-pietersap} 
provides all the information we need on the set
$\{p\equiv a\bmod d:\ord_p(\ell)=p-1 \}$, 
it suffices to consider the set
\[
\mc A_{d,a}^-:=\{p\equiv a\bmod d:\ord_p(\ell)=(p-1)/2 \},  
\]
which, under GRH, has density 
\begin{equation}
    \label{alphada-}
\alpha^{-}_{d,a} = \sum_{n=1}^\infty\frac{\varphi(d)\mu(n)c_a^-(n)}{[\Q(\zeta_{[d,2n]},\ell^{1/2n}):\Q]},
\end{equation}
where $c_a^-(n)=1$ if the automorphism $\sigma_a$ of $\Q(\zeta_d)$ determined by $\sigma_a(\zeta_d)=\zeta_d^a$ is the identity on the field $\Q(\zeta_d)\cap\Q(\zeta_{2n},\ell^{1/2n})$, and $c_a^-(n)=0$ otherwise.

\begin{Thm}\label{thm_Hminus}
Let $\ell$ be an odd prime, $a$ and $d$ coprime positive integers, $\delta_{d,a}^-$ as in \eqref{alphada-} 
and $\epsilon$ be arbitrary and fixed. Then for every $x$ sufficiently large we have
\[
\mc A_{d,a}^-(x)\leq (\alpha^{-}_{d,a}+\epsilon)\frac{x}{\varphi(d)\log x}, 
\]
where
\[ \alpha^{-}_{d,a}=A\,c^-(d,a)\,R(d,a), \]
with $R(d,a)$ as in \eqref{eq:rda}. If 
$4\mid d$, then
\begin{equation}
    \label{4doesnotdivide}
c^-(d,a) = \begin{cases}
\frac{1}{2}\left( 1-\frac{1}{\ell^2-\ell-1} \right) & \text{if } \ell\nmid d, a\equiv1\bmod4; \\
1-\left(\frac{-1}{\ell}\right)\frac{1}{\ell^2-\ell-1} & \text{if } \ell\nmid d, a\equiv3\bmod4; \\
0 & \textit{if } \ell\mid d, \left(\frac{\ell}{a}\right)=-1; \\
\frac{1}{2}\left( 3- \left(\frac{-1}{a}\right)\right) & \textit{if } \ell\mid d, \left(\frac{\ell}{a}\right)=1.
\end{cases} 
\end{equation}
If $4\nmid d$, then
\[ c^-(d,a) = \begin{cases}
\frac{3}{4}\left( 1-\frac{1}{\ell^2-\ell-1} \right) & \text{if } \ell\nmid d, \ell\equiv1\bmod4; \\
\frac{1}{4}\left(3+\frac{1}{\ell^2-\ell-1}\right) & \text{if } \ell\nmid d, \ell\equiv3\bmod4; \\
\frac{3}{4}\left( 1+\left(\frac{a}{\ell}\right) \right)
& \textit{if } \ell\mid d,\ell\equiv1\bmod4; \\
\frac{1}{8}\left( 5+ \left(\frac{a}{\ell}\right)\right) & \textit{if } \ell\mid d, \ell\equiv3\bmod4.
\end{cases} \]
Assuming GRH, we have
\[ \mc A_{d,a}^-(x)= \frac{\alpha^{-}_{d,a}}{\varphi(d)}
\frac{x}{\log x}+O_d\left( \frac{x\log\log x}{\log^2 x} \right). \]
\end{Thm}
\begin{Rem}
\label{rem:zero}
If $4\ell\mid d$ and $\kronecker{\ell}{a}=-1$, then 
$\alpha^-_{d,a}=0$ for trivial reasons and 
even $\mc A^{-}(d,a)$ is empty.
Namely, if $p\in \mc A^{-}(d,a)$,
then
$p\equiv a\bmod{\ell}$
and 
$\ord_p(\ell)=(p-1)/2$. 
It follows that
$\kronecker{\ell}{p}=1$ and  hence
$\kronecker{\ell}{a}=1$. 
\end{Rem}
\begin{Rem}
\label{rem:comparison}  
Note that $\mc P_{d,a}\subseteq \mc P^-_{d,a}$.
Comparison with \eqref{altset} shows that $\mc P^-_{d,a}=\mc P_{d,a}$ if 
$4\mid d$ and $a\equiv 3\bmod4$. Hence in this case, under GRH, we have
$c^-(d,a)=c(d,a)-c_1(d,a)$, with $c(d,a)$ and $c_1(d,a)$ explicitly given
in Theorem \ref{thm-ap}, respectively Theorem \ref{thm-pietersap}.
By definition, we have $\mc P^-_{d,a}=\mc P_{d,a}$ if and only if the density of primes $p$ such that $p\equiv a\bmod d$, $p\equiv1\bmod4$ and $\ord_p(\ell)=(p-1)/2$ is zero. This happens trivially if the two modular congruences are not compatible, namely when $4\mid d$ and $a\not\equiv1\bmod 4$. If they are compatible, then by Theorem \ref{thm_Hminus} under GRH the considered density is zero if and only if one of the following two conditions holds: \begin{itemize}
    \item $4\ell\mid d$, $a\equiv1\bmod4$ and $\kronecker{\ell}{a}=-1$;
    \item $4\nmid d$, $\ell\mid d$, $\ell=1\bmod 4$ and $\kronecker{a}{\ell}=-1$.
\end{itemize}
\end{Rem}

Combination of Theorem \ref{thm_Hminus} and Theorem \ref{thm-pietersap} yields the following result.
\begin{Thm}
\label{thm:deltada-}
We have
\[ \mathcal P^-_{d,a}(x)\leq (\delta^{-}_{d,a}+\epsilon)\frac{x}{\varphi(d)\log x}, \]
where $\delta^{-}_{d,a}=\alpha^-_{d,a}+\alpha_{d,a}=AR(d,a)c_2(d,a)$ with 
\[ 
c_2(d,a)=c^-(d,a)+c_1(d,a)=
\begin{cases}
\frac{1}{2}\left( 3+\frac{1}{\ell^2-\ell-1} \right) & \text{if } \ell\nmid d, \ell\equiv1\bmod4; \\
2 & \text{if } \ell\nmid d, \ell\equiv3\bmod4; \\
2 & \text{if } \ell\mid d, \ell\equiv1\bmod4; \\
\frac{1}{2}\left(3-\left(\frac{-1}{a}\right)\right) & \text{if } \ell\mid d, \ell\equiv3\bmod4,4\mid d; \\
\frac{3}{2} & \textit{if } \ell\mid d,\ell\equiv3\bmod4,4\nmid d.
\end{cases}
\]
Assuming GRH we have
\[ \mc P_{d,a}^-(x)= \frac{\delta^{-}_{d,a}}{\varphi(d)}\frac{x}{\log x}+O_d\left( \frac{x\log\log x}{\log^2 x} \right). \]
\end{Thm}
\subsubsection{Proofs}
We start by determining the coefficients $c^-_a(n)$ that occur in 
the infinite sum \eqref{alphada-}.
\begin{Lem}
\label{lem:fieldintersection}
Let $\ell$ be an odd prime,  $n$ a squarefree integer, and $d$ a natural number such that $4\mid d$. Let $\Delta$ be the discriminant of $\Q(\sqrt{\ell})$. We have
\[ \Q(\zeta_d)\cap\Q(\zeta_{2n},\ell^{1/2n})=\Q(\zeta_{(d,2n)},\alpha) \]
with
\[ \alpha=\begin{cases}
\sqrt{\ell} & \text{if~} \ell\mid d, \Delta\nmid 2n;\\
i & \text{if~} \ell\nmid d, \ell\mid n,2\nmid n, \ell\equiv3\bmod4;\\
1 & \text{otherwise}.
\end{cases} \]
In the first two cases $\Q(\zeta_{(d,2n)},\alpha)$ is a quadratic extension of $\Q(\zeta_{(d,2n)})$.
\end{Lem}

\begin{proof}
Let us set $I:=\Q(\zeta_d)\cap\Q(\zeta_{2n},\ell^{1/2n})$. By Kummer theory we may argue that, since $\Q(\zeta_\infty)\cap\Q(\zeta_{2n},\ell^{1/2n})$ is a finite abelian extension of $\Q(\zeta_{2n})$, it is of the form $\Q(\zeta_{2n},\ell^{e/2n})$ for some $e\geq0$, and hence by Schinzel's theorem \cite[Theorem 2]{schinzel} it is $\Q(\zeta_{2n},\sqrt{\ell})$. The latter extension equals $\Q(\zeta_{2n})$ if $\Delta\mid2n$.

Therefore, if $\ell\nmid dn$ or $\Delta\mid 2n$, then $I=\Q(\zeta_{(d,2n)})$. If $\Delta\nmid 2n$ and $\ell\mid d$, then noticing that $\Q(\zeta_{(d,2n)}) \subseteq I \subseteq \Q(\zeta_{(d,2n)},\zeta_{4\ell})$ and $\Q(\zeta_{4\ell})\subseteq\Q(\zeta_d)$, we deduce that $I=\Q(\zeta_{(d,2n)},\sqrt{\ell})$. 
If $\ell\mid n$, $n$ is odd, $\ell\equiv3\bmod4$ and $\ell\nmid d$, then we have $\Q(\zeta_{(d,2n)}) \subseteq I \subseteq \Q(\zeta_{(d,4n)})$, yielding $I=\Q(\zeta_{4(d,n)})$ as $\zeta_4\in I$ but $\zeta_4\notin \Q(\zeta_{(d,2n)})$.
\end{proof}

\begin{Cor}
\label{cor:c-}
Let $\ell$ be an odd prime,  $n$ a squarefree integer, and $a,d$ natural numbers such that $4\mid d$ and $(a,d)=1$. If  $a\not\equiv 1\bmod(d,2n)$, then $c^-_a(n)=0$. 
If $a\equiv 1\bmod(d,2n)$, then
\[ c^-_a(n)=\begin{cases}
\frac{1}{2}\left( 1+\left(\frac{\ell}{a}\right) \right) & \text{if~} \ell\mid d, \Delta\nmid n;\\
\frac{1}{2}\left( 1+\left(\frac{-1}{a}\right) \right) & \text{if~} \ell\nmid d, \ell\mid n,2\nmid n, \ell\equiv3\bmod4;\\
1 & \text{otherwise}.
\end{cases} \]
\end{Cor}
\begin{proof}
An immediate consequence of Lemma 
\ref{lem:fieldintersection} on noting that
$\sigma_a(\zeta_{(d,2n)})=\zeta_{(d,2n)}^a$, $\sigma_a(i)=i^a=\kronecker{-1}{a}i$, and, if $\ell\mid d$, then
$\sigma_a(\sqrt{\ell})=\kronecker{\ell}{a}\sqrt{\ell}$ (by Lemma \ref{lem:actiononsquare}).
\end{proof}
\begin{Rem} 
A different proof of Lemma 
\ref{lem:fieldintersection} is obtained on using 
that if $K/\Q$ is Galois, then
$$I:=[K\cap L:\mathbb Q]=\frac{[K:\mathbb Q][L:\mathbb Q]}{[K\cdot L:\mathbb Q]}.$$ 
Applying this equality with $K=\mathbb Q(\zeta_d)$ and $L=\mathbb Q(\zeta_{2n},\ell^{1/2n})$, computing all the degree occurring,
and using that $\varphi((d, 2n))\varphi([d,2n]) = \varphi(d)\varphi(2n)$,
then shows that in the first
two cases $I$ is a quadratic extension of
$\mathbb Q(\zeta_{(d,2n)})$ and $I=\mathbb Q(\zeta_{(d,2n)})$ otherwise. The proof is then easily completed.
\end{Rem}

\begin{proof}[Proof of Theorem \ref{thm_Hminus}]
Our starting point for is formula \eqref{alphada-}, which expresses $\alpha^-_{d,a}$ as an 
infinite sum, which we will rewrite as an Euler product 
using Lemma \ref{lem:Ssums} (the notation of which we 
will use).
Throughout $n$ will be a squarefree integer.
We first assume $4\mid d$, and so $[d,2n]=[d,n]$.
Recall that the degree $[\Q(\zeta_{[d,2n]},\ell^{1/2n}):\Q]$ equals $\varphi([d,n])n$ if  $\ell\mid[d,n]$ and $\varphi([d,n])2n$ otherwise.
We put
$$
\Sigma_1 = \sum_{a\equiv 1\bmod (d,2n)}\frac{\varphi(d)\mu(n)}{[\Q(\zeta_{[d,2n]},\ell^{1/2n}):\Q]},\quad
\Sigma_2=\sum_{a\equiv 1\bmod (d,2n)\atop \ell\mid d,\,\Delta\nmid n}\frac{\varphi(d)\mu(n)}{[\Q(\zeta_{[d,2n]},\ell^{1/2n}):\Q]},
$$
and 
$$\Sigma_3=\sum_{a\equiv 1\bmod (d,2n)\atop \ell\nmid d,\,\ell\mid n,\,2\nmid n,\,\ell\equiv 3\bmod{4}}\frac{\varphi(d)\mu(n)}{[\Q(\zeta_{[d,2n]},\ell^{1/2n}):\Q]}.$$
Notice that the three sums $\Sigma_2,\Sigma_3$ and $\Sigma_1$ (respectively)
reflect the three cases distinguished in Corollary \ref{cor:c-}.
Making the
values of $c^-_a(n)$ in \eqref{alphada-} explicit using Corollary \ref{cor:c-} we obtain
$$\alpha^{-}_{d,a} =\Sigma_1+\frac{1}{2}\Big(\kronecker{\ell}{a}-1\Big)\Sigma_2
+\frac{1}{2}\Big(\kronecker{-1}{a}-1\Big)\Sigma_3.$$ 

\emph{Case 1: $\ell\mid d$.}
Now $\Sigma_3=0$ and so
$$\alpha^{-}_{d,a} =\Sigma_1+\frac{1}{2}\Big(\kronecker{\ell}{a}-1\Big)\Sigma_2.$$ 
We have 
\[ \Sigma_1 = \sum_{ a\equiv 1\bmod (d,2n)}\frac{\mu(n)}{\omega_d(n)}, \]
which equals $S(1)=AR(d,a)$ for $a\equiv1\bmod4$, and 
equals $S(1)-S_2(1)=2S(1)=2AR(d,a)$ if $a\equiv3\bmod4$.

\emph{Subcase 1.1: $\kronecker{\ell}{a}=1$.} Then
$\alpha^{-}_{d,a} =\Sigma_1$ and we find $c^-(d,a)=\frac{1}{2}(3-\kronecker{-1}{a})$.

\emph{Subcase 1.2: $\kronecker{\ell}{a}=-1$.} 
Suppose $\ell\equiv 1\bmod{4}$. Then by quadratic reciprocity we have
$\kronecker{a}{\ell}=-1$ and so all $n$ that contribute
to $\Sigma_1$ satisfy $\ell\nmid n$. As $\Delta=\ell$ we infer
that $\Sigma_1=\Sigma_2$ and hence $\alpha^{-}_{d,a}=\Sigma_1-\Sigma_2=0$.
If $\ell\equiv 3\bmod{4}$, the condition $\Delta\nmid n$ 
is automatically 
satisfied and we obtain $\Sigma_2=\Sigma_1$, and hence again $\alpha^{-}_{d,a}=0$. (For a different argument why $\alpha^-_{d,a}=0$ see Remark \ref{rem:zero}.) We conclude
that $c^-(d,a)=0$.\\
\emph{Case 2: $\ell\nmid d$.}
Now
$$\alpha^{-}_{d,a}=\Sigma_1 
+\frac{1}{2}\Big(\kronecker{-1}{a}-1\Big)\Sigma_3.$$
We have
\[
   \Sigma_1  =
\Bigg(\sum_{\ell\mid n \atop a\equiv 1\bmod (d,2n)}+
		\frac{1}{2}\sum_{\ell\nmid n \atop a\equiv 1\bmod (d,2n)}
\Bigg) \frac{\mu(n)}{\omega_d(n)} .
\]
If $a\equiv1\bmod4$, then 
\[    
\Sigma_1 =   \frac{1}{2}\Bigg(
		\sum_{n\geq1  \atop a\equiv 1\bmod (d,n)}  +
		\sum_{\ell\mid n  \atop a\equiv 1\bmod (d,n)} \Bigg)\frac{\mu(n)}{\omega_d(n)}
		= \frac{1}{2}(S(1)+S(\ell))
		= \frac{A}{2}\left( 1-\frac{1}{\ell^2-\ell-1}\right)R(d,a),
\]
and hence
$c^-(d,a)=\frac{1}{2}( 1-\frac{1}{\ell^2-\ell-1})$.

Next suppose $a\equiv3\bmod4$. In this case in the density
sums we can restrict to odd $n$.\\
\emph{Subcase 2.1: $\ell\equiv 3\bmod4$.} 
Now
$$ \alpha^-_{d,a}=\Sigma_1-\Sigma_3=\frac{1}{2}\sum_{\ell\nmid n,\ 2\nmid n \atop a\equiv 1\bmod (d,n)}\frac{\mu(n)}{\omega_d(n)},$$
which is easily seen to equal
$$\frac{1}{2}(S(1)-S_2(1)-S(\ell)+
S_2(\ell))=S(1)-S(\ell)=A\left(1+\frac{1}{\ell^2-\ell-1}\right)R(d,a).$$
\emph{Subcase 2.2: $\ell\equiv 1\bmod4$.} 
Now $\Sigma_3=0$ and 
\[
    \alpha^-_{d,a}=\Sigma_1 =\Bigg(\sum_{\ell\mid n,\ 2\nmid n \atop a\equiv 1\bmod (d,n)}+\frac{1}{2}\sum_{\ell\nmid n,\ 2\nmid n \atop a\equiv 1\bmod (d,n)}\Bigg)\frac{\mu(n)}{\omega_d(n)},
\]
which equals
\[
(S(\ell)-S_2(\ell))+(S(1)-S(\ell))=
S(1)+S(\ell)=A\left(1-\frac{1}{\ell^2-\ell-1}\right)R(d,a).
\]
We conclude that $c^-(d,a)=1-\kronecker{-1}{\ell}\frac{1}{\ell^2-\ell-1}$.

This completes the proof in the case $4\mid d$. Suppose now that $4\nmid d$. We may argue as in 
Sect.\,\ref{alternative} and, keeping the 
notation $d,a_1,a_3$ as there, we obtain 
\[ \alpha_{d,a}^-=\frac{\alpha_{d_1,a_1}^-+\alpha_{d_1,a_3}^-}{2}=\frac{A}{2}(c^-(d_1,a_1)+c^-(d_1,a_3))R(d,a), \]
and hence
\[
 c^-(d,a)=\frac{c^-(d_1,a_1)+c^-(d_1,a_3)}{2}. 
\]
Since $\kronecker{a}{\ell}=\kronecker{a_1}{\ell}=\kronecker{a_3}{\ell}$, the idea is to rewrite 
\eqref{4doesnotdivide} using quadratic reciprocity in terms of $\kronecker{a}{\ell}$. This gives rise to 
more cases, but makes it easy to determine 
$c^-(d,a)$. For example, if $\ell\mid d,\,\kronecker{\ell}{a}=1$, and
$\ell\equiv 1\bmod{4}$, then $c^-(d_1,a_1)=1$ and $c^-(d_1,a_3)=2$, and
thus $c^-(d,a)=\frac{3}{2}$ if $\ell\mid d,\kronecker{a}{\ell}=1$, 
and $\ell\equiv 1\bmod{4}$.
\end{proof}

\subsection{The $H^+$-sequence}
Let $\ell$ be an odd prime number and $1\leq a < d$ be coprime integers.  In this section we consider the  set of rational (odd) primes
\[ \mc P_{d,a}^+:=\{p\equiv a\bmod d:
\ord_p(\ell)=p-1\text{~or~}\ord_p(\ell)\text{~is~odd} \}.   \]
This set can be written as a disjoint union of two sets:
$$\mc P_{d,a}^+=\mc A_{d,a}\cup \{p\equiv a\bmod d:\ord_p(\ell)\text{~is~odd}\}.$$
Under GRH the density of $\mc A_{d,a}$ is given in Theorem \ref{thm-pietersap}.
The density of the second set has been unconditionally determined 
in case $\ell$ is a positive 
rational number by Moree and Sury \cite{MS}. Unfortunately, even in the case where $\ell$ is an odd prime, there are many cases and for this reason we will not 
write the details out here. However, under this restriction on
$\ell$ the density is always positive (this follows from \cite[Theorem 5]{MS}).
If the 
discriminant of $\mathbb Q(\sqrt{\ell})$ divides $d$ and $\kronecker{\ell}{a}=1$, then the set
$\mc A_{d,a}$ is empty (see Remark \ref{triviaaldichtheidnul}) and we
obtain an unconditional asymptotic for $\mc P_{d,a}^+(x)$. A particular
easy case arises for $\ell=3,a=11$ and $d=12$. Then we
have $\mc P^+_{12,11}=\{p:p\equiv 11\bmod{12}\}$.
This basically is a claim Fermat made in 1641! He made some similar, but unfortunately
wrong ones, for the details see \cite{MS}.
\subsection{Counting prime divisors of the $H$-, $H^-$- and $H^+$-sequences}
Denote the sets of prime divisors in the section header by, respectively, $\mc Q_H$, ${\mc Q}_{H^-}$ and ${\mc Q}_{H^+}$. By Proposition \ref{prop:Hdivisor} and Lemma \ref{lem:wieferich} these sets are very closely related to $\mc P_{G}$, $\mc P_{H^-}$ and $\mc P_{H^+}$, respectively. Assuming that the 
associated Wieferich sets are $o(x/\log x)$ (see 
Sect.\,\ref{sec:wieferich}), we arrive at the following conjecture. 
\begin{Con}
Asymptotically we have ${\mc Q}_H(x)\sim \mc P_{G}(x)$, ${\mc Q}_{H^-}(x)\sim\mc P_{H^-}(x)$ and ${\mc Q}_{H^+}(x)\sim\mc P_{H^+}(x)$, 
where the asymptotic behavior of $\mc P_{G}(x),\mc P_{H^-}(x),\mc P_{H^+}(x)$ is given by \eqref{eq:Gconj}, \eqref{Hmincon}, respectively \eqref{Hpluscon}.
\end{Con}
% punctuation checked manually until here

\section{Earlier work on the Genocchi $\ell$-integers}
\label{Genocchil}
We recapitulate some earlier work on $\ell$-Genocchi integers. None of it is directly relevant for the proofs
presented in this paper, and so it can be regarded as background reading.
\subsection{The case $\ell=2$}\label{subsec:Earlier2}
The original Genocchi numbers are obtained 
on taking $\ell=2$. These numbers have received considerable attention in the
literature.
It follows from \eqref{generating} and $H_n=G_n/2n$ that  
$$
\frac{2t}{e^t+1} = \sum_{n=1}^{\infty} G_n \frac{t^n}{n!}. 
$$
It is well-known that $G_1=1$, $G_{2n+1}=0$ for $n \ge 1$, and that $(-1)^nG_{2n}$ is an odd positive 
integer. 
%see Section \ref{subsec:Earlier2} for some further information.
%This is an easy consequence of \eqref{vsclausen} and Fermat's little theorem and can also be proven
%combinatorially, see Han and Liu \cite{HL}.  A combinatorial interpretation of 
\par Dumont \cite{Dumont} showed that $|G_{2n}|$ equals the number of permutations $p$ of 
$\{1,2,\ldots,2n-1\}$ such that $p(i)< p(i+1)$ for $p(i)$ odd and 
$p(i)> p(i+1)$ for $p(i)$ even $(1\le i\le 2n-2)$. For a survey of related
material see Stanley \cite{Stanley}.
\par Setting $\ell=2$ in 
\eqref{Voronoi}, in case $p-1\nmid 2n$  we 
obtain 
\begin{equation}
    \label{Voronoi2}
H_{2n}\equiv -2^{2n-1}\sum_{j=(p+1)/2}^{p-1}j^{2n-1}\equiv 2^{2n-1}\sum_{j=1}^{(p-1)/2}j^{2n-1}\pmod*{p}.
\end{equation}
A related result is due to Emma Lehmer \cite{EmmaLehmer}, who showed in case $2n\not\equiv 2\bmod{(p-1)}$ that
$$H_{2n}\equiv -\sum_{j=1}^{(p-1)/2}(p-2j)^{2n-1}\bmod{p^2}.$$
Under the same assumption on $n$ and in the same spirit she showed that
$$\sum_{j=1}^{[p/4]}(p-4j)^{2n-1}\equiv -H_{2n}(2^{2n-1}+1)\bmod{p^2},\quad p>3.$$
In case $2n\equiv 2\bmod{(p-1)}$ both congruences still hold true modulo $p$.
\par For $|t|<\pi/2$ we have (see, e.g., \cite[Prop.\,1.17]{AIK})
$$\tan t =\sum_{n=1}^{\infty}T_n\frac{t^{2n-1}}{(2n-1)!},\,\text{~with~}\,T_n=(-4)^n H_{2n}.$$
The numbers $T_n$ are called  \emph{tangent numbers} 
and count the number of all alternating permutations of length $2n-1$ (see 
Entringer \cite{Entringer} or Knuth and Buckholtz \cite{KB}).
\par Let $p>3$ be a prime satisfying $p\equiv 3\pmod*{4}$ and 
put $m=(p-1)/2$. Then the class number $h(-p)$ of the quadratic field $\mathbb Q(\sqrt{-p})$ satisfies
$$h(-p)= \frac{1}{2-\kronecker{2}{p}}\sum_{j=1}^m \kronecker{j}{p}\equiv \frac{1}{2-2^m}\sum_{j=1}^mj^m\equiv \frac{2^{-m}}{2-2^m}
H_{\frac{p+1}{2}}
\equiv -2B_{\frac{p+1}{2}}\,\pmod*{p},$$
a congruence due to Cauchy. 
The first identity is an easy consequence of Dirichlet's class
number formula (see, for example, \cite[p.\,99]{AIK}),
the second congruence follows from \eqref{Voronoi2} on setting $n=(p+1)/4$.
The first identity implies that
$h(-p)\le (p-1)/2$ and thus the congruence uniquely determines
the value of $h(-p)$. 
For some related results we refer to
the recent preprint by Miná\v{c} et al.\,\cite{Minacetal}.

\subsection{The case $\ell$ is odd}\label{subsec:Earlierodd}
Despite the enormous literature on variations of Bernoulli numbers, we found only very little earlier
work on $\ell$-Genocchi numbers for
odd $\ell$. 
For example, in case $\ell=3$ and $p>3$, Emma Lehmer \cite{EmmaLehmer} showed that
$$H_{2n}\equiv -2\sum_{j=1}^{[p/3]}(p-3j)^{2n-1} \bmod{p^2}.$$
The deepest result we found gives a connection with functions related
to polylogarithms. Namely, given any integer $k$ consider the formal series
\[ l_k(s) = \sum_{n=1}^\infty \frac{s^n}{n^k}. \]
W\'ojcik \cite{wojcik} found an explicit formula for $l_k$ for $k\leq 0$, namely
\[ l_k(s) = -\frac{sR_n(s)}{(s-1)^{n+1}}, \]
where $n=-k$ and the $R_n\in\mathbb Z[s]$ are the classical Euler-Frobenius polynomials defined by the formula
\[ \frac{1-s}{e^t-s} = \sum_{n=0}^\infty \frac{R_n(s)}{(1-s)^n}\frac{t^n}{n!}. \]
The individual terms in the first sum of \eqref{generating} can be connected with the series $l_k(s)$, 
namely it can be shown (cf.\,Urbanowicz and Williams \cite[p.\,132]{UW}) that
\[ \frac{z}{e^t-z} = \sum_{n=1}^\infty (-1)^nl_{-n}(z)\frac{t^n}{n!}. \]
From \eqref{generating} we then infer on equating Taylor coefficients
\[ H_n=(-1)^n\sum_{a=1}^{\ell-1} l_{1-n}(\zeta_\ell^a). \]

%\vfil\eject
\section{Experimental data}
In this section we provide some numerical examples for the densities of $G$-, $H^+$- and $H^-$-irregular primes, and $G$-irregular primes in arithmetic progressions and compare them with our conjectural predictions. As customary, $\pi(x)$ is the number of primes $p\leq x$. All data have been produced with SageMath \cite{Sage}.
\begin{table}[h!] 
\centering
\caption{The ratio $\mc P_G(x)/\pi(x)$ for $x=10^5$}
\label{tabG1}
\vspace{0.2cm}
\begin{tabular}{|c|c|c|}
\hline
$\ell$    & experimental & theoretical  \\ \hline \hline  
 2 & 0.661593 & 0.659776 \\
 3 & 0.635113  & 0.637095 \\ 
 5 & 0.657214 & 0.653807  \\ 
 7 & 0.660863  &  0.657010 \\ 
 11 & 0.660133  &  0.658736 \\ 
 13 & 0.659612  & 0.659045 \\ 
 17 & 0.662948  & 0.659358 \\ 
 19 & 0.657110  & 0.659444 \\  \hline
\end{tabular}
\end{table}

\begin{table}[h!] 
\centering
\caption{The ratios $\mc P_{H^+}(x)/\pi(x)$ and $\mc P_{H^-}(x)/\pi(x)$ for $x=10^5$}
\label{tabHpm}
\vspace{0.2cm}
\begin{tabular}{|c|cc||cc|}
\hline
& $\mc P_{H^+}(x)/\pi(x)$ & & $\mc P_{H^-}(x)/\pi(x)$ & \\ \hline
$\ell$    & experimental & theoretical & experimental & theoretical \\ \hline \hline  
 2 & 0.599145  & 0.596279 & 0.603315 & 0.603072 \\ 
 3 &  0.568390 & 0.571007 & 0.588198 & 0.591731 \\ 
 5 &  0.563699 & 0.559070 & 0.599458 & 0.600088 \\ 
 7 & 0.575271  & 0.571007 & 0.604671 & 0.601689 \\ 
 11 &  0.571726 & 0.571007 & 0.604462 & 0.602552 \\ 
 13 &  0.571518 & 0.569544 & 0.600292 & 0.602706 \\ 
 17 & 0.573499  & 0.570170 & 0.607173 & 0.602863 \\ 
 19 &  0.569537 & 0.571007 & 0.599875 & 0.602906 \\   \hline
\end{tabular}
\end{table}
\FloatBarrier

\begin{table}[h!] 
\centering
\caption{The ratio $\mc P_G(d,a)(x)/\pi(x;d,a)$ for $x=10^5$}
\label{tabG3}
\vspace{0.2cm}
\begin{tabular}{|c|cc|cc|}
\hline
$\ell$  & $d$ & $a$  & experimental & theoretical  \\ \hline \hline  
 3 & 3 & 1 & 0.816097 & 0.818547 \\
  & 3 & 2 & 0.454337 & 0.455642\\ 
  & 4 & 1 & 0.717473 & 0.727821\\
  & 4 & 3 & 0.552752 & 0.546368\\ \hline
 5 & 3 & 1 & 0.725396 & 0.723046\\
  & 3 & 2 & 0.589241 & 0.584569\\ 
  & 4 & 1 & 0.763970 & 0.761246\\
  & 4 & 3 & 0.550459 & 0.546368\\ \hline
 7 & 3 & 1 & 0.723311 & 0.725608\\
  & 3 & 2 & 0.598624 & 0.588412\\ 
  & 4 & 1 & 0.764178 & 0.767652\\
  & 4 & 3 & 0.557548 & 0.546368\\  \hline
\end{tabular}
\qquad
\begin{tabular}{|c|cc|cc|}
\hline
$\ell$  & $d$ & $a$  & experimental & theoretical  \\ \hline \hline  
 11 & 7 & 1 & 0.695580 & 0.700353 \\
  & 7 & 2 & 0.661802 & 0.650412 \\
  & 7 & 3 & 0.649291 & 0.650412 \\
  & 7 & 4 & 0.649917 & 0.650412 \\
  & 7 & 5 & 0.670559 & 0.650412 \\
  & 7 & 6 & 0.634279 & 0.650412 \\ \hline
   & 15 & 1 & 0.753962 & 0.770096 \\
  & 15 & 2 & 0.576314 & 0.568929 \\  
  & 15 & 4 & 0.706422 & 0.712620 \\
   & 15 & 7 & 0.724771 & 0.712620 \\
  & 15 & 8 & 0.573812 & 0.568929 \\
    & 15 & 11 & 0.649708 & 0.655144 \\
   & 15 & 13 & 0.712260 & 0.712620 \\
  & 15 & 14 & 0.583820 & 0.568929 \\  \hline
\end{tabular}
\end{table}
\FloatBarrier

\section*{Acknowledgement}
The authors thank Karl Dilcher, Bakir Farhi, 
Masanobu Kaneko and Bernd Kellner for kindly answering some
Bernoulli number related questions. 
A very helpful e-mail exchange with Efthymios Sofos led to
a simplified proof of Proposition \ref{casesbounds}.
Thanks to Antonella Perucca for her constant support and encouragement.

\end{document}